\documentclass[10pt]{article}
\usepackage{fullpage,epsfig}
\usepackage{enumerate}
\usepackage{amsmath}
\usepackage{stmaryrd}
\usepackage{amsthm}
\usepackage{amssymb}
\usepackage{latexsym}
\usepackage{accents}
\usepackage{mathrsfs}
\usepackage{subfigure}
\usepackage{color}
\usepackage[colorinlistoftodos]{todonotes}

\numberwithin{equation}{section}

\newtheorem{example}{Example}[section]
\newtheorem{theorem}[example]{Theorem}
\newtheorem{proposition}[example]{Proposition}
\newtheorem{definition}[example]{Definition}
\newtheorem{lemma}[example]{Lemma}
\newtheorem{corollary}[example]{Corollary}

\newtheorem{remark}[example]{Remark}

\def\x{{\it x}}

\newcommand{\R}{{\mathbb R}}
\newcommand{\Grad}{\mathcal{GY}^{+\infty,-\infty}(\O;\R^{n\times n})}
\newcommand{\Gradrho}{\mathcal{GY}_\varrho^{+\infty,-\infty}(\O;\R^{n\times n})}

\newcommand{\Rm}{{\R^{n\times n}}}
\newcommand{\N}{{\mathbb N}}
\renewcommand{\det}{{\rm det}\, }

\newcommand{\reg}{{\R^{2\times 2}_{\rm inv}}}
\newcommand{\regplus}{{\R^{2\times 2}_{\rm inv+}}}

\newcommand{\be}{\begin{eqnarray}}

\newcommand{\ee}{\end{eqnarray}}

\renewcommand{\d}{{\rm d}}

\newcommand{\ra}{\right\rangle}

\newcommand{\la}{\left\langle}
\newcommand{\md}{{\rm d}}

\renewcommand{\O}{\Omega}

\newcommand{\A}[1]{\langle#1\rangle}
\newcommand{\Rn}{\R^{n}}

\newcommand{\rca}{{\rm rca}}

\newcommand{\Rrho}{{R_\varrho^{2 \times 2}}}
\newcommand{\RrhoPlus}{{R_{\varrho+}^{2 \times 2}}}
\newcommand{\wstar}{\stackrel{*}{\rightharpoonup}}

\renewcommand{\Gradrho}{\mathcal{GY}_\varrho^{\infty,-\infty}(\O;\R^{2\times 2})}
\newcommand{\GradPlus}{\mathcal{GY}_+^{\infty,-\infty}(\O;\R^{2\times 2})}

\newcommand{\EEE}{\color{black}}

\newcommand{\ONE}{\color{black}}
\newcommand{\WE}{\color{black}}

\title{Characterization of gradient Young measures generated by homeomorphisms in the plane}

\begin{document}

\begin{center}
{\huge \bf  Characterization of gradient Young measures generated by homeomorphisms in the plane}\\
\vspace{3mm}
Barbora Bene\v{s}ov\'{a}$^{1,2}$ \& Martin Kru\v{z}\'{\i}k$^{3,4}$ \\
\vspace{2mm}
$^1$ Department of Mathematics I, RWTH Aachen University, D-52056 Aachen, Germany\\
$^2$ Institute for Mathematics, University of W\"{u}rzburg, Emil-
Fischer-Stra\ss e 40, D-97074 W\"{u}rzburg, Germany\\
$^2$Institute of Information Theory and Automation of the CAS, Pod vod\'{a}renskou
v\v{e}\v{z}\'{\i}~4, CZ-182~08~Praha~8, Czech Republic \\
$^3$ Faculty of Civil Engineering, Czech Technical
University, Th\'{a}kurova 7, CZ-166~ 29~Praha~6, Czech Republic

\bigskip

\bigskip
\bigskip
\begin{abstract}
We  characterize Young measures generated by  gradients of bi-Lipschitz  orientation-preserving  maps in the plane.  This question is motivated by variational problems in nonlinear elasticity where the orientation preservation and injectivity of the admissible deformations are key requirements. These results enable us  to derive new weak$^*$ lower semicontinuity results for integral functionals depending on  gradients. As an application,  we show the existence of a minimizer for an integral functional with nonpolyconvex energy density among bi-Lipschitz homeomorphisms.       
\end{abstract}
\end{center}

\section{Introduction} 
 The aim of this paper is to describe oscillatory properties of sequences of gradients of bi-Lipschitz maps in the plane that preserve the orientation, i.e., the gradients of which have a positive determinant. Such mappings naturally appear in non-linear hyperelasticity where they act as \emph{deformations}.

Although there are more general definitions of a {deformation}, i.e.\ a function $y: \Omega \to \R^n$ that maps each point in the reference configuration to its current position, we confine ourselves to the one by P.G.~Ciarlet \cite[p.~27]{ciarlet} which requires injectivity in the domain $\Omega\subset\R^n$, sufficient smoothness and orientation preservation. Here, ``sufficient smoothness'' will mean that  a considered deformation will be a  \emph{homeomorphism} in order to prevent cracks or cavitation and its (weak) deformation gradient will be integrable, i.e.\ $y\in W^{1,p}(\O;\R^n)$ with $1 < p \le +\infty$.

Clearly, a deformation is an invertible map but, in our modeling, we put an additional requirement on $y^{-1}$
---namely, it should again qualify as a deformation, which is motivated by the fact that we aim to model \emph{the elastic response of the specimen}. In the elastic regime, the specimen returns to its original shape after all loads are released and so, since the r\^{o}les of the reference and the deformed configuration can be exchanged, we would like to understand the releasing of loads as applying a new loading, inverse to the original one, in the deformed configuration and the ``return'' of the specimen as the corresponding deformation. Thus, we define the following set of deformations
\begin{align}
W^{1,p,-p}_+(\Omega;\mathbb{R}^n) = \Big\{&y: \Omega \mapsto y(\Omega) \text{ an orientation preserving homeomorphism}; \nonumber \\ &y\in W^{1,p}(\O;\R^n) \text { and } y^{-1}\in W^{1,p}(y(\Omega);\R^n) \Big\}\ .
\label{deformations}
\end{align}

Although invertibility of deformations is a fundamental requirement in elasticity it is still often omitted in modeling due to the lack of appropriate mathematical tools \WE to handle it. \EEE However, let us mention that some  ideas of incorporating invertibility  of the deformation already appeared e.g.~in \cite{ball81,ciarlet-necas,fonseca-ganbo,currents,tang,MullerQi,MullerSpector,Henao} and very recently e.g.~in \cite{iwaniec1,daneri-pratelli}.

Stable states of the specimen are found by minimizing
\be\label{motivation}
J(y)=\int_\O W(\nabla y(x))\,\md x\ ,
\ee
where $W \colon \R^{n\times n}\to\R$ is the stored energy density, i.e.\ the potential of the first Piola-Kirchhoff stress tensor, over the set of admissible deformations \eqref{deformations}; possibly with respect to a Dirichlet boundary condition $y = y_0$ on $\partial \Omega$.

A natural, still open, question is under which minimal conditions on a continuous $W$ satisfying  $W(A)=+\infty$ if $\det A\le 0$ and 
\begin{equation}\label{blowup}
W(A) \to +\infty \qquad \text{whenever} \qquad \det\, A \to 0_+ 
\end{equation}
we can guarantee that $J$ is weakly lower-semicontinuous on  \eqref{deformations}. In fact, Problem 1 in  Ball's paper \cite{ball-puzzles}: ``\emph{Prove the existence of energy minimizers for elastostatics for quasiconvex stored-energy functions satisfying \eqref{blowup}}'' is closely related. 

Here we answer this question for the special case of bi-Lipschitz mappings in the plane; i.e.\ we restrict our attention to the setting $p=\infty, n=2$. It is natural to conjecture that the sought equivalent characterization of weak* lower semicontinuity will lead to a suitable notion of quasiconvexity. We confirm this conjecture and show that $J$ is weakly* lower semicontinuous on $W^{1,\infty,-\infty}_+(\Omega;\mathbb{R}^2)$ if and only if it is \emph{bi-quasiconvex} in the sense of Definition \ref{biqc-Def}.


\begin{remark}[Quasiconvexity]
We say that $W$ is quasiconvex if
\be\label{quasiconvexity} |\O|W(A)\le\int_\O W(A+\nabla\varphi(x))\,\md x\ \ee 
holds for all $\varphi\in W^{1,\infty}_0(\O;\R^n)$ and all $A\in\R^{n\times n}$  \cite{morrey}. It is well known \cite{dacorogna} that if $W$ takes only finite values and is quasiconvex then $J$  in \eqref{motivation} is   weakly* lower semicontinuous on $W^{1,\infty}(\Omega;\mathbb{R}^n)$ and so, in particular, also on $W^{1,\infty,-\infty}_+(\Omega;\mathbb{R}^2)$. 

Nevertheless, as we shall see, classical quasiconvexity is too restrictive in the bi-Lipschitz setting; indeed, since we narrowed the set of deformations it can be expected that a larger class of energies will lead to weak* lower semicontinuity of $J$. This can be also understood from a mechanical point of view: quasiconvex materials are described by energies having the property that among all \emph{deformations} with affine boundary data the affine ones are stable. Thus, since we now restricted the set of deformations it seems natural to verify \eqref{quasiconvexity} only for bi-Lipschitz functions; this is indeed the sought after convexity notion which we call bi-quasiconvexity (cf. Def. \ref{biqc-Def}).
\end{remark}

To prove our main result, we completely and explicitly characterize gradient Young measures generated by sequences in  $W^{1,\infty,-\infty}_+(\Omega;\mathbb{R}^2)$ (cf. Section \ref{results}). Young measures extend the notion of solutions from Sobolev mappings to parametrized measures \cite{ball3,fonseca-leoni,pedregal,r,schonbek,tartar,tartar1,y}. The idea is to describe the limit behavior of $\{J(y_k)\}_{k\in\N}$ along a minimizing sequence $\{y_k\}_{k\in\N}$. Actually, one needs to work with the so-called gradient Young measures because it is the gradient of the deformation entering the energy in \eqref{motivation}. Their explicit characterization  is due to Kinderlehrer and Pedregal \cite{k-p1,k-p}; however, it does not take into account any constraint on determinants or invertibility of the generating mappings. In spite of this drawback, gradient Young measures are massively used in literature to model solid-to-solid phase transitions as appearing in, e.g., shape memory alloys; cf.~\cite{ball-james1,mueller,kruzik-luskin,pedregal,r}. 

 Yet, not excluding matrices with a negative determinant may add non-realistic phenomena to the model. Indeed, it is well-known that the modeling of solid-to-solid phase transitions via Young measures is closely related to the so-called quasiconvex envelope of $W$ which must be convex along rank-one lines, i.e.\ lines whose elements differ by a rank-one matrix. Not excluding matrices with negative determinants, however, adds many non-physical rank-one lines to the problem. Notice, for instance,  that  {\it any} element of SO$(2)$ is on a rank-one line with {\it any} element of $\mathrm{O}(2)\setminus{\rm SO}(2)$. Consequently, the determinant must inevitably change its sign on such line. 




 The first attempt to include constraints on the sign of the determinant of the generating sequence appeared in \cite{quasiregular} where quasi-regular generating sequences in the plane were considered; however injectivity of the mappings could only be treated in the homogeneous case. Then, in \cite{bbmkgpYm} the characterization of gradient Young measures generated by sequences whose gradients are invertible matrices for the case where gradients as well as their inverse matrices are bounded in the $L^\infty$-norm was given. Very recently, Koumatos, Rindler, and Wiedemann \cite{krw} characterized Young measures generated by orientation preserving maps in $W^{1,p}$ for $1<p<n$; however they did not account for the restriction that deformations should be injective. 

 Therefore, this contribution (to our best knowledge) presents the first characterization of Young measures that are generated by sequences that are orientation-preserving  and \emph{globally invertible} and so qualify to be admissible deformations in elasticity.

Generally speaking, the main difficulty in characterizing sets of Young measures generated by deformations (or, at least, mappings having constraints on the invertibility and/or determinant of the deformation gradient) is that this constraint is \emph{non-convex}. Thus, many of the standardly used techniques such as smoothing by a mollifier kernel are not applicable. In our context, we need to be able to modify the generating sequence on a vanishingly small set near the boundary to have the same boundary conditions as the limit; i.e.\ to construct a cut-off technique. It can be seen from \eqref{quasiconvexity}, that standard proofs of characterizations of gradient Young measures \cite{k-p1,k-p} or weak lower semicontinuity of quasiconvex functionals \cite{dacorogna} \emph{will rely} on such techniques since the test functions in \eqref{quasiconvexity} have fixed boundary data. Usually, the cut-off is realized by convex averaging which is, of course, ruled out here. Novel ideas in \cite{bbmkgpYm,krw} are to solve differential inclusions near the boundary to overcome this drawback.  This allows to impose restrictions on the determinant of the generating sequence in several ``soft-regimes''; nevertheless, such techniques have not been generalized to more rigid constraints like the global invertibility.

Here we follow a different approach and, for bi-Lipschitz mappings in the plane, we obtain the result by exploiting bi-Lipschitz extension theorems \cite{daneri-pratelli-extension,tukia-ext}. Thus, by following a strategy inspired by \cite{daneri-pratelli} we modify the generating sequence (on a set of gradually vanishing measure near the boundary) first on a one-dimensional grid and then extend it.  The main reason why we confine ourselves to the bi-Lipschitz case and do not work in $W^{1,p,-p}_+(\Omega;\mathbb{R}^2)$ with $p< \infty$ is the fact that our technique relies on the extension theorem or, in other words, a full characterization of traces of bi-Lipschitz functions. To our best knowledge, such a characterization is at the moment completely open in $W^{1,p,-p}_+(\Omega;\mathbb{R}^2)$ with $p< \infty$. Still, let us point out its importance for finding minimizers of $J$ over \eqref{deformations}: in fact, constructing an extension theorem allows to precisely characterize the set of Dirichlet boundary data admissible for this problem. Notice that this question appears also in the existence proof for polyconvex materials and usually one assumes there that the set of admissible deformations is nonempty; \cite{ciarlet}.

\begin{remark}[Growth conditions] 
Even though in this paper we restrict our attention to bi-Lipschitz functions, let us point under which growth of the energy we can guarantee that the minimizing sequence of $J$ lies in $W^{1,p,-p}_+(\Omega;\mathbb{R}^n)$. Namely, it follows from the works of J.M. Ball \cite{ball77, ball81} that it suffices to require that $W$ is finite only on the set of matrices with positive determinant and (``cof'' stands for the cofactor in dimension 2 or 3)
\begin{equation}
C\Big(|A|^p + \frac{1}{\det\,A} + \frac{|\mathrm{cof}(A)|^p}{\det\, A^{p-1}} - 1\Big) \leq W(A) \leq C\Big(|A|^p + \frac{1}{\det\, A} + \frac{|\mathrm{cof}(A)|^p}{\det\, A^{p-1}} + 1\Big),
\label{growth-improved}
\end{equation}
as well as fix suitable boundary data (for example bi-Lipschitz ones).\footnote{As pointed above, since the traces of functions in $W^{1,p,-p}_+(\Omega;\mathbb{R}^2)$ are not precisely characterized to date, it is hard to decide what ``suitable boundary data'' are. In any case, in the plane bi-Lipschitz boundary data are sufficient.}

Polyconvexity, i.e.\ convexity in all minors of $A$, is fully compatible with such growth conditions (they are themselves polyconvex) whence if $W$ is polyconvex minimizers of \eqref{motivation} \emph{over $W^{1,p}(\O;\R^n)$, $p>n$} are indeed deformations;  i.e.\ are globally invertible and elements of $W^{1,p,-p}_+(\Omega;\mathbb{R}^n)$. We refer, e.g.~, to \cite{ciarlet,dacorogna} for various generalizations  of this result. However, while polyconvexity is a sufficient condition it is not  a \emph{necessary one}.

On the other hand, classical results on quasiconvexity yielding existence of minimizers \cite{dacorogna} are compatible with neither the growth conditions proposed in this remark nor \eqref{blowup}. In fact, existence of a minimizer of \eqref{motivation} on $W^{1,p}(\O;\R^n)$ for quasiconvex $W$ can be, to date, proved only if 
\be\label{gr-cond}
c(-1+|A|^p)\le W(A)\le \tilde c(1+|A|^p)\ .
\ee

The reason why the current proofs of existence of minimizers for quasiconvex cannot be extended to \eqref{growth-improved} is exactly the non-convexity detailed above.
\end{remark}

The plan of the paper is as follows. We first introduce necessary definitions and tools in Section~\ref{Start}. Then we state the main results in 
Section~\ref{results}. Proofs are postponed to Section~\ref{Proofs} while the novel cut-off technique is presented in Section \ref{sect-cutOff}.

\bigskip
 
\section{Preliminaries}\label{Start}
Before stating our main theorems in Section~\ref{results}, let us summarize, at this point, the notation as well as background information that we shall use later on. 
 
We define the following subsets of the set of invertible matrices:
\begin{align}
\label{Rvarrho} \Rrho&=\{A\in\R^{2\times 2}\mbox{ invertible};\ |A^{-1}|\leq \varrho\ \&\  |A| \leq \varrho \}\ , \\
\label{Rvarrho+ }\!\!\!\! \RrhoPlus&=\{A\in \Rrho;\  {\rm det}\ A>0\}\,
\end{align}
for $1\leq\varrho<\infty$. Note that both $\Rrho$ and  $\RrhoPlus$  are compact.
 Set
$$
\R^{2 \times 2}_\mathrm{inv} = \bigcup_\varrho \Rrho \qquad \qquad 
\R^{2 \times 2}_\mathrm{inv+} = \bigcup_\varrho \RrhoPlus.
$$

We assume that the matrix norm used above is sub-multiplicative, i.e.\ that $|AB|\leq|A||B|$ for all $A,B\in\R^{2\times 2}$  and such that the norm of the identity matrix  is one.
This means that if $A\in\RrhoPlus$ then $\min(|A|,|A^{-1}|)\geq 1/\varrho$.

\begin{definition}
A mapping $y:\O\to\R^2$ is called L-bi-Lipschitz (or shortly bi-Lipschitz) if there is $L\ge 1$ such that for all $x_1,x_2\in\O$ 
\begin{equation}\label{bi-li}
\frac{1}{L}|x_1-x_2|\le |y(x_1)-y(x_2)|\le L|x_1-x_2|\ .
\end{equation}
The number $L$ is called the bi-Lipschitz constant of $y$.
\end{definition}

This means that $y$ as well as its inverse $y^{-1}$ are Lipschitz continuous, hence $y$ is homeomorphic. Notice that $\frac{1}{L}\le|\nabla y(x)|\le L$ for almost all $x\in\O$.  

\begin{definition}
We say that $\{y_k\}_{k\in\N}\subset W^{1,\infty,-\infty}_+(\O;\R^2)$ is bounded in $W^{1,\infty,-\infty}_+(\O;\R^2)$ if the bi-Lipschitz constants of $y_k$, $k\in\N$, are uniformly bounded and $\{y_k\}_{k\in\N}$ is bounded in $W^{1,\infty}(\O;\R^n)$. Moreover, we say that $y_k \wstar y$ in $W^{1,\infty,-\infty}_+(\O;\R^2)$ if the sequence is bounded and $y_k \wstar y$ in $W^{1,\infty}(\O;\R^2)$.
\end{definition}

We would like to stress the fact that $W^{1,\infty,-\infty}_+(\O;\R^2)$ is not a linear function space.

\begin{remark}\label{convergence}

Notice that if $y_k \wstar y$ in $W^{1,\infty,-\infty}_+(\O;\R^2)$, we can give a precise statement on how the inverses of $ \{y_k\}$ converge if the target domain is fixed throughout the sequence; i.e.\ if $y_k: \Omega \to \widetilde{\Omega}$ for all $k \in \mathbb{N}$. This can be achieved for example by fixing Dirichlet boundary data through the sequence.

In such a case it is easy to see that $y_k^{-1} \wstar y^{-1}$ in $W^{1,\infty}(\widetilde{\Omega}, \Omega)$: Since the gradients of the inverses $\nabla y_k^{-1}$ are uniformly bounded by the uniform bi-Lipschitz constants, we may select at least a subsequence converging weakly* in $W^{1,\infty}(\widetilde{\Omega}, \Omega)$ and thus strongly in $L^{\infty}(\widetilde{\Omega}, \Omega)$. Nevertheless, the latter allows us to pass to the limit in the identity $y_k^{-1}(y_k(x)) = x$ for any $x \in \Omega$ and therefore  to identify the weak* limit as $y^{-1}$; in other words, the weak limit is identified independently of the selected subsequence which assures that the whole sequence $ \{y_k^{-1}\}_{k\in \mathbb{N}}$ converges weakly* to $y^{-1}$.
\end{remark}

\bigskip

Let us now summarize the theorems on invertibility, extension from the boundary in the bi-Lipschitz case and on  approximation by smooth functions needed  below.

\begin{theorem}[Taken from \cite{ball81}]\label{jmball}
Let $\O\subset\R^n$ be a bounded  Lipschitz domain. Let $u_0:\overline{\O}\to\R^n$ be continuous in $\overline{\O}$ and one-to-one in $\O$ such that $u_0(\O)$  is also bounded and  Lipschitz. Let $u\in W^{1,p}(\O;\R^n)$ for some $p>n$,  $u(x)=u_0(x)$ for all $x\in\partial\O$, and let $\det\nabla u>0$ a.e.~in $\O$. Finally, assume that  for some $q>n$
\begin{equation}
\int_\O|(\nabla u(x))^{-1}|^q\det\nabla u(x)\,\md x<+\infty\ .
\label{integral-Cond}
\end{equation}
Then $u(\overline{\O})=u_0(\overline{\O})$ and $u$ is a homeomorphism of $\O$ onto $u_0(\O)$. Moreover, the inverse map $u^{-1}\in W^{1,q}(u_0(\O);\R^n)$ and  $\nabla u^{-1}(z)=(\nabla u(x))^{-1}$ for $z=u(x)$ and a.a.~$x\in\O$. 
\end{theorem} 

\begin{remark} Let us point out that the original statement of the Theorem~\ref{jmball} requires that $u_0(\O)$ satisfies the so-called cone condition and that $\O$ is strongly Lipschitz. These conditions hold if $\O$ and $u_0(\O)$ are  bounded and Lipschitz domains; cf.~\cite[p.~83-84]{adams-fournier}.  
\end{remark}

\begin{theorem}[Square bi-Lipschitz extension theorem due to \cite{daneri-pratelli-extension} and previously \cite{tukia-ext}]
\label{extensionTheorem}
There exists a geometric constant $C\leq 81\cdot 63600$ such that every $L$ bi-Lipschitz map $u: \partial \mathcal{D}(0,1) \mapsto \mathbb{R}^2$ (with $\mathcal{D}(0,1)$ the unit square) admits a $C L^4$ bi-Lipschitz extension $v: \mathcal{D}(0,1) \mapsto \Gamma$ where $\Gamma$ is the bounded closed set such
that $\partial \Gamma = u(\partial \mathcal{D}(0,1))$.
\end{theorem}

\begin{remark}[Rescaled squares]
Let us note, that the theorem above holds with the \emph{same geometric constant} $C$ also for rescaled squares
$\mathcal{D}(0, \epsilon)$ with some $\epsilon > 0$, possibly small. Indeed, for $u: \partial \mathcal{D}(0, \epsilon) \mapsto \mathbb{R}^2$, we define the rescaled function $\tilde{u}: \partial \mathcal{D}(0,1) \mapsto \mathbb{R}^2$ through $\tilde{u}(x) = \epsilon u(x/\epsilon)$; note that both functions have the same bi-Lipschitz constant. This function is then extended to obtain $\tilde{v}: \mathcal{D}(0,1) \mapsto \mathbb{R}^2$ as in the above theorem. Again we rescale $\tilde{v}$, under preservation of the bi-Lipschitz constant, to $v: \mathcal{D}(0, \epsilon) \mapsto \mathbb{R}^2$ $v=\frac{1}{\epsilon} \tilde v(\epsilon x)$. So, $v$ is $C L^4$ bi-Lipschitz and, since $\tilde{u}$ coincides with $\tilde{v}$ on the boundary of the unit square, $v$ coincides with $u$ on $ \partial \mathcal{D}(0, \epsilon)$.
\end{remark}

\begin{theorem}[Smooth approximation \cite{iwaniec1} and in the bi-Lipschitz case also by \cite{daneri-pratelli}]
\label{SmoothApprox}
Let $\O\subset\R^2$ be bounded open  and $y\in W^{1,p}(\Omega; \R^2)$ ($1 < p < \infty$)  be an orientation preserving homeomorphism. Then it can be, in the $W^{1,p}$-norm, approximated by diffeomorphisms having the same boundary value as $y$. Moreover, if $y$ is bi-Lipschitz, then there exists a sequence of diffeomorphisms $\{y_k\}$ having the same boundary value as $y$ and $y_k$, $y_k^{-1}$ approximate $y$, $y^{-1}$ in $W^{1,p}$-norm with $1 < p < \infty$, respectively.
\end{theorem}

\subsection{Young measures}
We denote by ``$\rca(S)$'' the set of Radon measures on a set $S$. Young measures on a bounded  domain $\O\subset\Rn$ are weakly* measurable mappings
$x\mapsto\nu_x:\O\to \rca(\R^{n\times n})$ with values in probability measures; the adjective ``weakly* measurable'' means that,
for any $v\in C_0(\R^{n\times n})$, the mapping
$\O\to\R:x\mapsto\A{\nu_x,v}=\int_{\R^{n\times n}} v(s)\nu_x(\d s)$ is
measurable in the usual sense. Let us remind that, by the Riesz theorem,
$\rca(\R^{n\times n})$, normed by the total variation, is a Banach space which is
isometrically isomorphic with $C_0(\R^{n\times n})^*$, where $C_0(\R^{n\times n})$ stands for
the space of all continuous functions $\R^{n\times n}\to\R$ vanishing at infinity.
Let us denote the set of all Young measures by ${\cal Y}(\O;\R^{n\times n})$. It
is known (see e.g.~\cite{r}) that ${\cal Y}(\O;\R^{n\times n})$ is a convex subset of $L^\infty_{\rm
w}(\O;\rca(\R^{n\times n}))\cong L^1(\O;C_0(\R^{n\times n}))^*$, where the subscript ``w''
indicates the aforementioned property of weak* measurability  Let $S\subset\R^{n\times n}$ be a compact set. A classical result
\cite{tartar} states that for every sequence $\{Y_k\}_{k\in\N}$
bounded in $L^\infty(\O;\R^{n\times n})$ such that $Y_k(x)\in S$ there exists a subsequence (denoted by
the same indices for notational simplicity) and a Young measure
$\nu=\{\nu_x\}_{x\in\O}\in{\cal Y}(\O;\R^{n\times n})$ satisfying
\be\label{jedna2}
\forall v\in C(S):\ \ \ \ \lim_{k\to\infty}v(Y_k)=\int_{\R^{n\times n}}v(s)\nu_x(\d s)\ \ \ \ \
\ \ \mbox{ weakly* in }L^\infty(\O)\ .
\ee
Moreover, $\nu_x$ is supported on $\bar S$  for almost all $x\in\O$.  On the other hand, if $\mu=\{\mu_x\}_{x\in\O}$, $\mu_x$ is supported on $S$ for almost all $x\in\O$ and $x\mapsto\mu_x$ is weakly* measurable then there exist a sequence  $\{Z_k\}_{k\in\N}\subset L^\infty(\O;\R^{n\times n})$, $Z_k(x)\in S$  and \eqref{jedna2} holds with $\mu$ and $Z_k$ instead of $\nu$ and $Y_k$, respectively.

Let us denote by ${\cal Y}^\infty(\O;\R^{n\times n})$ the set of all Young measures which are created in this way, i.e.\ by taking all bounded sequences in $L^\infty(\O;\R^{n\times n})$. Moreover, we denote by ${\cal GY}^\infty(\O;\R^{n\times n})$ the subset of ${\cal Y}^\infty(\O;\R^{n\times n})$ consisting of measures generated by gradients of $\{y_k\}_{k\in\N}\subset W^{1,\infty}(\O;\R^n)$, i.e.\  $Y_k=\nabla y_k$ in \eqref{jedna2}. The following result  is due to Kinderlehrer and Pedregal \cite{k-p1,k-p} (see also \cite{mueller,pedregal}):  

\begin{theorem}[adapted from \cite{k-p1,k-p}]
Let $\Omega$ be a bounded Lipschitz domain. Then the parametrized measure $\nu \in {\cal Y}^\infty(\O;\R^{n\times n})$ is in ${\cal GY}^\infty(\O;\R^{n\times n})$ if and only if
\begin{enumerate}
  \item  there exists $z \in W^{1,\infty}(\Omega; \R^n)$ such that $\nabla z(x) = \int_{\Rm} A \nu_x(\d A)$  for a.e. \ $x \in \Omega$,
  \item $\psi(\nabla z(x)) \leq \int_{\Rm}\psi(A) \nu_x(\d A) $ for a.e. \ $x \in \Omega$ and for all $\psi$ quasiconvex, continuous and bounded from below,
  \item $\mathrm{supp} \ \nu_x \subset K$ for some compact set $K \subset \Rm$ for a.e. \ $x \in \Omega$.
\end{enumerate}
\end{theorem}

\bigskip

%
%
%
\section{Main results}
\label{results}
 
We shall denote, for $\varrho \geq 1$, 
\begin{align*}
\Gradrho = \big\{ \nu &\in \mathcal{Y}^\infty(\Omega; \R^{2 \times 2}) \\ &\text { that are generated by $\varrho$-bi-Lipschitz, orientation preserving maps} \big\},
\end{align*}
and 
$$
\GradPlus = \bigcup_{\varrho \geq 1} \Gradrho.
$$
As already pointed out in the introduction we seek for an explicit characterization of $\GradPlus$; it can be expected that, 
when compared to \cite{k-p1}, we shall restrict the support of the Young measure as in \cite{quasiregular,bbmkgpYm,krw} but also alter the Jensen inequality by changing the notion of quasiconvexity.

\begin{definition}
\label{biqc-Def}
Suppose $v:\R^{2\times 2}\to\R\cup\{+\infty\}$ is bounded from below and  Borel measurable. Then we denote  
$$
Z v(A)=\inf_{\varphi\in W^{1,\infty, -\infty}_{A}(\O;\R^2)}|\O|^{-1}\int_\O v(\nabla \varphi(x))\,\md x\ ,
$$
with
$$
W^{1,\infty,-\infty}_{A} (\O;\R^2) = \begin{cases}
\Big\{y \in W^{1,\infty, -\infty}_+(\O;\R^2); y(x)=Ax \text{ if  $x\in\partial \Omega$} \Big\} & \text{if $\det\,A > 0$,} \\
\emptyset &\text{else.}
\end{cases}
$$
and say that $v$ is bi-quasiconvex on $\R^{2 \times 2}_\mathrm{inv+}$ if $Zv(A) = v(A)$ for all $A \in \R^{2 \times 2}_\mathrm{inv+}$. 
Here
we set $\inf \emptyset=+\infty$.
\end{definition}

\begin{remark} \mbox{}
\label{biQCProp}
\begin{enumerate}
\item  Notice that actually $Zv(A)\le v(A)$ if det $A>0$  and $ Zv(A) = +\infty$ otherwise, so that  $Zv\not\le v$ in general. Moreover, the infimum in the definition of $Zv(A)$ is, generically, not attained.\item  Any $v$ as in Definition \ref{biqc-Def} bi-quasiconvex  if and only if
\be\label{bi-def}
|\O|v(A)\le\int_\O v(\nabla\varphi(x))\,\md x\ 
\ee 
for all $\varphi \in W^{1,\infty,-\infty}_+(\Omega;\mathbb{R}^2)$, $\varphi = Ax$ on $\partial \Omega$ and all $A\in\R^{2 \times 2}_\mathrm{inv+}$.  Indeed, clearly if v is bi-quasiconvex then \eqref{bi-def} holds. On the other hand, if \eqref{bi-def} holds, we have that $v(A)\le Zv(A)$ for $A\in \R^{2 \times 2}_\mathrm{inv+}$ by taking the infimum in \eqref{bi-def}. Moreover,  $Zv(A)\le v(A)$ for such $A$, so that $Zv(A)=v(A)$. 
\item We recall that the condition of bi-quasiconvexity is less restrictive than the usual quasiconvexity and there obviously exist bi-quasiconvex functions on $\R^{2\times 2}$ which are not quasiconvex (for example, take $v:\R\to\R$ with $v(0)=1$ and $v(A)=0$ if $A\ne 0$.). Also, we can allow for the growth \eqref{blowup}.
\item  It is interesting to investigate whether, for any $v$ as from Definition \ref{biqc-Def}, $Zv(A)$ is already a bi-quasiconvex function. If one wants to follow the standard approach known from the analysis of classical quasiconvex function \cite{dacorogna}, this consists in showing that $Zv$ can be actually replaced by $Z'v$ defined through  
$$
Z' v(A)=\inf_{\varphi\in W^{1,\infty, -\infty}_{A}(\O;\R^2)\text{ piecewise affine }}|\O|^{-1}\int_\O v(\nabla \varphi(x))\,\md x\ ,
$$
and that the latter is bi-quasiconvex. To do so, one relies on the density of piecewise affine function which, in our case, is available through Theorem \ref{SmoothApprox}. Moreover, to employ the density argument, one needs to show that $Z'v$ is rank-1 convex on $\R^{2 \times 2}_\mathrm{inv+}$ and hence continuous. This is done by constructing a sequence of faster and faster oscillating laminates that are altered near the boundary to meet the boundary condition. Now, since an appropriate cut-off technique becomes available through this work, it seems that this approach should be feasible. Nevertheless, the details are beyond the scope of the present paper and we leave them for future work.

Let us remark that an alternative to the above methods may be possible along the lines of the recent work \cite{Conti}.
\end{enumerate}
\end{remark}

 The main result of our paper is the following characterization theorem.
 \begin{theorem}\label{THM1}
Let $\Omega \subset \R^2$ be a bounded Lipschitz domain. Let $\nu\in \mathcal{Y}^\infty(\O;\R^{2\times 2})$. Then  $\nu \in \GradPlus$ if and only if the following three conditions hold:
 \begin{gather}
 \exists \varrho \geq 1 \text{ s.t. }\, \mathrm{supp}\, \nu_x\subset \RrhoPlus \mbox{ for a.a.~$x\in\O$}\ , \label{supp} \\
    \exists\ u\in W^{1,\infty,-\infty}_+(\O;\R^2)\ :\  \nabla u(x)=\int_{\regplus} A \nu_x(\d A)\  \ ,\label{firstmoment0} \\
    \intertext{$\exists \bar{c}(\varrho) > \varrho$ such that for a.a.~$x\in\O$, all  $\tilde\varrho\in[\bar{c}(\varrho);+\infty]$, and all  $v\in \mathcal{O}(\tilde\varrho)$  the following  inequality is valid}
    Z v(\nabla u(x))\le\int_{\regplus} v(A)\nu_x(\md A)\ , \label{qc0}
 \end{gather}
with
\begin{equation}
\mathcal{O}(\varrho)=\{ v:\R^{2\times 2}\to\R\cup\{+\infty\};\ v\in C(\Rrho)\ ,\ v(A)=+\infty \mbox { if $A\in\R^{2\times 2}\setminus\RrhoPlus$}\}\ .
\label{Orho}
\end{equation}

\end{theorem}

An easy corollary is the following:
\begin{corollary}\label{wlsc}
Let $\Omega \subset \R^2$ be a bounded Lipschitz domain. Let $v$ be in $\mathcal{O}(+\infty)$. Let $\{y_k\}_{k \in \N} \subset W^{1,\infty,-\infty}_+(\Omega;\mathbb{R}^2)$  and   suppose that $y_k \wstar y$ in $W^{1,\infty,-\infty}_+(\Omega;\mathbb{R}^2)$. Then $v$ is bi-quasiconvex if and only if $y\mapsto I(y)=\int_\O v(\nabla y(x))\,\d x$  is sequentially weakly* lower semicontinuous with respect to the convergence above.
\end{corollary}

Finally, as an application we can state the following statement about the existence of  minimizers.
\begin{proposition}\label{minimizer}
Let $\Omega \subset \R^2$ be a bounded Lipschitz domain and let $0\le v \in\mathcal{O}(+\infty)$ be bi-quasiconvex.
Let further $\varepsilon>0$ and define $I_\varepsilon:W^{1,\infty,-\infty}_+(\O;\R^2)\to\R$
$$
I_\varepsilon (u)= \int_\O v(\nabla u(x))\,\d x +\varepsilon(\|\nabla u\|_{L^\infty(\O;\R^{2\times 2})}+\|\nabla u^{-1}\|_{L^\infty(u(\O);\R^{2\times 2})})\ .$$

Let $u_0\in W^{1,\infty,-\infty}_+(\O;\R^2)$ and  
$$\mathcal{A}=\{ u\in W^{1,\infty,-\infty}_+(\O;\R^2);\,  u=u_0\mbox{ on }\partial\O\}\ .$$

Then there is a minimizer of $I_\varepsilon$ on $\mathcal{A}$.

\end{proposition}

\begin{remark} \mbox{}
\begin{enumerate}
\item Note that, we needed in Theorem \ref{THM1} that  $\tilde\varrho>\varrho$ so that boundedness of $\int_\Omega v(\nabla y_k) \md x$ does not yield the right $L^\infty$-constraint of the gradient of the minimizing sequence. This is actually a known fact in the $L^\infty$-case \cite{k-p1} and is usually overcome by assuming that the generating sequence does not need to be Lipschitz but is only bounded in some $W^{1,p}(\Omega; \R^2)$ space. Alternatively, one can use Proposition~\ref{minimizer} stated above.

\item It will follow from the proof that the constant $\bar{c}(\varrho)$ is actually determined by the extension Theorem \ref{extensionTheorem}.

 \item Note that if one can show that $Zv$ is already a bi-quasiconvex function (cf. Remark \ref{biQCProp}(4)) then \eqref{qc0} can be replaced by requiring that 
\begin{equation}
v(\nabla u(x))\le\int_{\regplus} v(A)\nu_x(\md A)
\label{qc0-alt}
\end{equation}
is fulfilled for all bi-quasiconvex $v$ in $\mathcal{O}(\tilde\varrho)$. Indeed, \eqref{qc0-alt} follows directly from \eqref{qc0} if $v$ is bi-quasiconvex. On the other hand, if \eqref{qc0-alt} holds and if we knew that $Zv$ is bi-quasiconvex, we know that
$$
Zv(\nabla u(x))\le\int_{\regplus} Zv(A)\nu_x(\md A) \leq \int_{\regplus} v(A)\nu_x(\md A)\ ,
$$
where the second inequality is due to Remark \ref{biQCProp}(1). 

\end{enumerate}
\end{remark}


\bigskip


 \section{Proofs}\label{Proofs}
 Here we prove Theorem~\ref{THM1}. Actually, we follow in large parts \cite{k-p1,pedregal} since, as pointed out in the introduction, the main difficulty lies in constructing an appropriate cut-off which we do in Section \ref{sect-cutOff}; so, we mostly just sketch the proof and refer to these references.
 
\subsection{Proof of Theorem \ref{THM1} - Necessity}

Condition \eqref{supp} follows from \cite[Propositions~2.4 and 3.3]{bbmkgpYm} and from the fact that any Young measure generated by  a sequence bounded in the $L^\infty$ norm is supported on a compact set.

In order to show \eqref{firstmoment0}, realize that it expresses the fact that the first moment of  $\nu$ is just the weak* limit of a generating sequence $\{\nabla y_k\}\subset L^\infty(\O;\R^{2\times 2})$.  The sequence $\{y_k\}$ is also  bounded  in $W_+^{1,\infty,-\infty}(\O;\R^2)$ and 
$\{y_k\}$ converges strongly to some $y\in W^{1,\infty}(\O;\R^2)$. Passing to the limit in \eqref{bi-li} written for $y_k$ instead of $y$ shows that $y$ is bi-Lipschitz. The  $L^\infty$- weak* convergence of $\det\nabla y_k$ to $\det\nabla y$ finally implies that  $y\in W_+^{1,\infty,-\infty}(\O;\R^2)$ as a bi-Lipschitz map cannot change sign of its Jacobian on $\O$.

To prove \eqref{qc0} we follow a standard strategy, e.g., as in \cite{pedregal}. First, we show that almost every  individual measure $\nu_x$ is a homogeneous Young measure generated by bi-Lipschitz maps with affine boundary data. The latter fact is implied by Theorem~\ref{cut-off}.
Then \eqref{qc0} stems from the very definition of bi-quasiconvexity.

 \begin{lemma}\label{localization}
Let $\nu\in \Gradrho$. Then $\mu=\{\nu_a\}_{x\in\O}\in\Gradrho$ for a.e. \ $a\in \O$.
\end{lemma}

\bigskip
\noindent
{\it Proof.}  Note that the construction in the proof of \cite[Th.~7.2]{pedregal} does not affect orientation-preservation nor the bi-Lipschitz property. Namely, if gradients of a bounded sequence $\{u_k\}\subset W_+^{1,\infty,-\infty}(\O;\R^2)$ generate $\nu$ then  for almost all $a\in\O$  one constructs a localized sequence $\{ju_k(a+x/j)\}_{j,k\in\N}$ (note that this function is clearly injective if $u_k$ was; since the norm of the gradient is just shifted this yields the bi-Lipschitz property) whose gradients generate $\mu$ as $j,k\to\infty$. 
\hfill $\Box$
 
\bigskip

\begin{proposition}\label{proposition:jensen}
Let $\nu\in\GradPlus$, supp$\,\nu\subset\Rrho$ be  such that   $\nabla y(x)= \int_\Rrho A\nu_x(\md A)$ for almost all $x\in\O$, where $y\in W^{1,\infty,-\infty}_+(\O;\R^2)$. Then for all $\tilde\varrho\in[\bar{c}(\varrho);+\infty]$, almost all $x\in\O$  and all $v\in\mathcal{O}(\tilde\varrho)$ we have 
\be \int_\reg v(A)\nu_x(\md A)\ge Z v(\nabla y(x))\ .\ee
\end{proposition}

\bigskip

\noindent{\it Proof.} 
We know from Lemma~\ref{localization} that $\mu=\{\nu_a\}_{x\in\O}\in\Gradrho$ for a.e.~$a\in\O$, so there exits its generating sequence $\{\nabla u_k\}_{k\in\N}$ such that $\{u_k\}_{k\in\N}\subset W_+^{1,\infty,-\infty}(\O;\R^2)$ and for almost all $x\in\O$ and all $k\in\N$ $\nabla u_k(x)\in\Rrho$.  Moreover, $\{u_k\}_{k\in\N}$  weakly* converges to the map $x \mapsto (\nabla y(a))x$ which is bi-Lipschitz.   

Using Corollary~\ref{corollary-cut-off}, we can, without loss of generality, suppose that $u_k$ is $\tilde{\varrho}$-bi-Lipschitz for all $k\in\N$ and $u_k(x)=\nabla y(a)x$ if $x\in\partial\O$. Therefore, we have
$$
|\O|\int_\regplus v(A)\nu_a(\md A) = \lim_{k\to\infty} \int_\O v(\nabla u_k(x))\,\md x \ge |\O|Z v(\nabla y(a)) \ .$$
\hfill
$\Box$
  
\bigskip

\subsection{Proof of Theorem~\ref{THM1} - sufficiency}
                                                               
We  need to show  that conditions \eqref{supp},\eqref{firstmoment0}, and \eqref{qc0} are also sufficient for $\nu \in\mathcal{Y}^\infty(\O;\R^{2\times 2})$ to be in ${\mathcal{GY}^{+\infty,-\infty}(\O;\R^{2\times 2})}$. Put 
\be\mathcal{U}^\varrho_A=\{y\in W_{A}^{1,\infty,-\infty}(\O;\R^2);\  \nabla y(x)\in \RrhoPlus\,\mbox{for a.a.~$x\in\O$}\}\ ;\ee
In other words this is the set of $\varrho$-bi-Lipschitz functions with affine boundary values equal to $x\mapsto Ax$. 
Consider for $A\in\R^{2\times 2}_\mathrm{inv}$ the set 
\be\mathcal{M}^\varrho_A=\{\overline{\delta_{\nabla y}};\ y\in\mathcal{U}^\varrho_A\}\ ,\ee
where $\overline{\delta_{\nabla y}}\in \rca(\R^{2\times 2})$ is defined \WE for all $v\in C_0(\R^{2\times 2})$ \EEE as $\la \overline{\delta_{\nabla y}}, v\ra=|\O|^{-1}\int_\O v(\nabla y(x))\,\md x$; $\overline{\mathcal{M}^\varrho_A}$ will denote its weak$^*$ closure.

\begin{lemma}\label{convexity}
 Let $A\in\RrhoPlus$.  Then the set $\mathcal{M}^\varrho_A$ is nonempty and  convex.
\end{lemma}

\begin{proof}
To show that $\mathcal{M}^\varrho_A\ne\emptyset$ is trivial because $x\mapsto y(x)=Ax$ is an element of this set as  $A$ has a positive determinant. 

To show that $\mathcal{M}^\varrho_A$ is convex we follow \cite[Lemma~8.5]{pedregal}.  We take $y_1,y_2\in \mathcal{U}^\varrho_A$ and, for a given $\lambda\in(0,1)$, we find a subset $D\subset\O$ such that $|D|=\lambda|\O|$. There are two countable disjoint families  of subsets of $D$ and $\O\setminus D$ of the form 
$$
\{a_i+\epsilon_i\O;\ a_i\in D,\ \epsilon_i>0,\ a_i+\epsilon_i\O\subset D\}$$
and 
$$
\{b_i+ \rho_i\O;\ b_i\in \O\setminus D,\ \rho_i>0,\ b_i+\rho_i\O\subset\O \setminus D\}\ $$
such that 
$$
D=\bigcup_{i}(a_i+\epsilon_i\O)\cup N_0 \ , \qquad \O\setminus D=\bigcup_{i}(b_i+\rho_i\O)\cup N_1  \ ,
$$
where the Lebesgue measure of $N_0$ and $N_1$ is zero. 
We define 
$$
y(x)=
\begin{cases}
\epsilon_iy_1\left(\frac{x-a_i}{\epsilon_i}\right)+Aa_i & \mbox{ if $x\in a_i+\epsilon_i\O$, }\\
 \rho_iy_2\left(\frac{x-b_i}{\rho_i}\right)+Ab_i & \mbox{ if $x\in b_i+\rho_i\O$, }\\
 Ax &\mbox{ otherwise,}
\end{cases}
\quad \mbox{yielding} \quad
\nabla y(x)=
\begin{cases}
\nabla y_1\left(\frac{x-a_i}{\epsilon_i}\right)& \mbox{ if $x\in a_i+\epsilon_i\O$, }\\
 \nabla y_2\left(\frac{x-b_i}{\rho_i}\right) & \mbox{ if $x\in b_i+\rho_i\O$, }\\
 A &\mbox{ otherwise.}
\end{cases}
$$
We must show that $y$ is $\varrho$-bi-Lipschitz; actually, as $\nabla y(x) \in \RrhoPlus$ a.e., we only need to check the injectivity of the mapping. 

To this end, we apply Theorem~\ref{jmball}. Notice that \eqref{integral-Cond} clearly holds for any $q \in (1, \infty)$ due to the a.e. \ bounds on $\nabla y$. Moreover, we have affine boundary data, $y(x)=Ax$, so that indeed the boundary data form a homeomorphism and, since $\Omega$ was a bounded Lipschitz domain, so will be $A\Omega=\{Ax;\, x\in\O\}$. Thus we conclude that, indeed, $y$ is $\varrho$-bi-Lipschitz.
 
In particular, $y\in{\mathcal U}_A^\varrho$ and 
$\overline{\delta_{\nabla y}}=\lambda\overline{\delta_{\nabla y_1}}+(1-\lambda)\overline{\delta_{\nabla y_2}}\ .$
\end{proof}

The following homogenization lemma can be proved the same way as  \cite[Th.~7.1]{pedregal}. The argument showing that a generating sequence of $\overline{\nu}$
comes from bi-Lipschitz orientation preserving maps comes from Theorem~\ref{jmball} the same way as in the proof of Lemma~\ref{convexity}.  

\bigskip 

\begin{lemma}\label{homogenization}
Let  $\{u_k\}_{k\in\N}\subset W^{1,\infty, -\infty}_A(\Omega;\R^2)$ be a bounded sequence in $W^{1,\infty,-\infty}_+(\O;\R^2)$. 
Let the Young measure $\nu\in \GradPlus$ be generated by $\{\nabla u_k\}_{k\in\N}$. Then there is a another bounded sequence $\{w_k\}_{k\in\N}\subset W^{1,\infty, -\infty}_A(\Omega;\R^2)$ that generates a homogeneous (i.e. independent of $x$)  measure  $\bar\nu$   defined through
\be\label{homog} 
\int_{\RrhoPlus} v(s)\bar\nu(\md s)= \frac{1}{|\O|}\int_{\O}\int_{\RrhoPlus} v(s)\nu_x(\md s)\,\md x\ , \ee
for any $v\in C(\RrhoPlus)$ and almost all $x\in\O$. Moreover, $\bar\nu \in \GradPlus$.
\end{lemma}

\bigskip

\begin{proposition}\label{homocase}
Let $\mu$ be a probability measure supported on a compact set  $K\subset \R^{2\times 2}_{\alpha+}$ for some $\alpha \geq 1$  and let  $ A=\int_K s\mu(\md s)$.  Let $\varrho>\alpha$ and let 
\be\label{jensen} 
Z v(A)\le \int_{K} v(s)\mu(\md s)\ ,\ee
for all $v\in \mathcal{O}(\varrho)$.
Then $\mu\in \GradPlus$ and it is generated by gradients of mappings from $\mathcal{U}^\varrho_A$.
\end{proposition}

\bigskip

\begin{proof}
   First, notice that $|A|\le \alpha<\varrho<+\infty $. Secondly, the set of measures $\mu$ in the statement of the proposition is convex and contains  $\mathcal{M}^\varrho_A$ as its convex and non void  subset due to Lemma~\ref{convexity}. We show that no fixed $\mu$ satisfying \eqref{jensen} can be separated from the weak* closure of $\mathcal{M}^\varrho_A$ by a hyperplane.  We argue by a contradiction argument. Then  by the Hahn-Banach theorem, assume \ONE that there is  $\tilde v\in C_0(\R^{2\times2})$ \EEE that separates $\mathcal{M}^\varrho_A$ from $\mu$. In other words, there exists a constant $\tilde{c}$  such that
$$
\la \nu,\tilde v\ra \geq \tilde{c} \text{  for all $\nu\in  \mathcal{M}^\varrho_A$} \qquad \text{and} \qquad 
\la \mu,\tilde v\ra < \tilde{c} 
$$

However, since we are working with  probability measures, we may  use $\tilde{v}-\tilde c$ instead of  $\tilde{v}$.  In this  way, we can put  $\tilde{c} =0$. Hence,  without loss of generality, we  assume that
$$
0\le \la \nu,\tilde v\ra=\int_\Rrho \tilde v(s)\nu(\md s) =|\O|^{-1}\int_\O \tilde v(\nabla y(x))\,\md x\ , $$
for all $\nu \in \mathcal{M}^\varrho_A$ (and hence all $y\in\mathcal{U}_A^\varrho$)
 and $0>  \la \tilde \mu,\tilde v\ra$. 
Now, the function 
$$
 v(F) = \begin{cases} \tilde {v}(F) &\text{if $F \in \RrhoPlus$},\\
+\infty &\text{else},
\end{cases}
$$
is in $\mathcal{O}(\varrho)$. Notice that it follows  from (\ref{jensen}) that $Z  v(A)$ is finite.
Thus, $Z v(A)=\inf_{\mathcal{U}_A^\varrho}|\O|^{-1}\int_\O v(\nabla y(x))\,\md x$. Hence, $Z v(A) \geq 0$ and, by \eqref{jensen}, $0\le Zv(A)\le \int_{K} v(s)\mu(\md s)=\int_{K} \tilde v(s)\mu(\md s)$. \ONE As this holds for all hyperplanes,  $\mu\in \overline{\mathcal{M}^\varrho_A}$, a contradiction. \EEE As \ONE $C_0(\R^{2\times 2})$  \EEE is separable, the weak* topology on bounded sets in \ONE its dual, $\rca(\R^{2\times 2})$, \EEE  is   metrizable. 
Hence, there is a sequence $\{u_k\}_{k\in\N}\subset \mathcal{U}_A^\varrho$  such that for all $v\in C(\RrhoPlus)$ (and all $v\in\mathcal{O}(\varrho)$)
\begin{eqnarray}\label{lim1}
\lim_{k\to\infty}\int_\O v(\nabla u_k(x))\,\md x= |\O|\int_{\RrhoPlus} v(s)\mu(\md s)\ ,\end{eqnarray}
and $\{u_k\}_{k\in\N}$ is bounded in $W^{1,\infty,-\infty}_+(\O;\R^{2\times 2})$.  
Let $\nu$ be a Young measure generated by $\{\nabla u_k\}$ (or a subsequence of it). Then we have for  $v$ as above
\begin{eqnarray}\label{lim2}
\lim_{k\to\infty}\int_\O v(\nabla u_k(x))\,\md x =\int_\O\int_{\RrhoPlus} v(s)\nu_x(\md s)\,\md x=|\O|\int_{\RrhoPlus} v(s)\mu(\md s) \ .
\end{eqnarray}
As $u_k(x)=Ax$ for $x\in\partial\O$  we apply Lemma~\ref{homogenization} to get a new sequence $\{\tilde u_k\}$ bounded in $W^{1,\infty,-\infty}_+(\O;\R^{2\times 2})$ with $\tilde u_k(x)=Ax$ for $x\in\partial\O$. The sequence $\{\nabla \tilde u_k\}$ generates a homogeneous Young measure $\bar\nu$ given by \eqref{homog}, so that  in view of \eqref{lim2} we get for $g\in L^1(\O)$
 $$
\lim_{k\to\infty}\int_\O g(x)v(\nabla \tilde u_k(x))\,\md x= \int_\O g(x)\,\md x\frac{1}{|\O|}\int_{\O}\int_{\RrhoPlus} v(s)\nu_x(\md s)\,\md x
=\int_\O\int_{\RrhoPlus} g(x)v(s)\mu(\md s)\,\md x\ .$$
\end{proof}

\begin{lemma}\label{auxiliary} (see \cite[Lemma~7.9]{pedregal} for a more general case)
Let $\O\subset\R^n$ be an open domain  with $|\partial\O|=0$ and
let $N\subset\O$ be of the zero Lebesgue measure. For
$r_k:\O\setminus N\to (0,+\infty)$ and $\{f_k\}_{k\in\N}\subset
L^1(\O)$ there exists a set of points $\{a_{ik}\}\subset
\O\setminus N$ and positive numbers $\{\epsilon_{ik}\}$,
$\epsilon_{ik}\le r_k(a_{ik})$ such that
$\{a_{ik}+\epsilon_{ik}\bar\O\}$ are pairwise disjoint for each
$k\in\N$, $\bar\O=\cup_i \{a_{ik}+\epsilon_{ik}\bar\O\}\cup N_k$
with $|N_k|=0$ and for any $j\in\N$ and any $g\in L^\infty(\O)$
$$
\lim_{k\to\infty}\sum_i
f_j(a_{ik})\int_{a_{ik}+\epsilon_{ik}\O}g(x)\,\md x= \int_\O
f_j(x)g(x)\,\md x\ .$$

\end{lemma}

\ONE
In fact,  the points $\{a_{ik}\}$ can be chosen from the intersection of sets of   Lebesgue points of all $f_j$, $j\in\N$. Notice that  this intersection has
the full Lebesgue measure. Here for each $j\in\N$, $f_j$ is identified with its precise representative \cite[p.~46]{evans-gariepy}. We adopt this identification below whenever we speak about a value of an integrable function at a  particular point. 

\EEE

\
\bigskip

\noindent {\it Proof of Theorem~\ref{THM1} - sufficiency.} \mbox{}
Some parts of the proof follow 
\cite[Proof of Th.~6.1]{k-p1}. We are looking for a sequence $\{u_k\}_{k\in\N}\subset
W^{1,\infty,-\infty}_+(\O;\R^2)$  satisfying
$$
\lim_{k\to\infty} \int_\O v(\nabla u_k(x))g(x)\,\md
x=\int_{\O} \int_{\R^{2\times 2
n}}v(s)\nu_x(\md s)g(x)\,\md x\  $$ for
all $g\in \Gamma$ and any $v\in S$, where
$\Gamma$ and $S$ are countable dense subsets of $C(\bar\O)$ and
$C(\RrhoPlus)$, respectively.

First of all notice that, as $u\in W^ {1,\infty,-\infty}_+(\O;\R^2)$ from \eqref{firstmoment0} is differentiable in $\O$ outside a set of measure zero called $N$, we may find for every $a\in\O\setminus N$ and every $k > 0$ some $1/k>r_k(a)>0$ such that for any $0<\epsilon < r_k(a)$ we have for every $y\in\O$
\be\label{derivative}
\frac1\epsilon | u(a+\epsilon y)-u(a)-\epsilon \nabla u(a)y|\le \frac1k\ .
\ee
Applying Lemma~\ref{auxiliary} and using its notation, we can find $a_{ik}\in\O\setminus N$, $\epsilon_{ik}\le  r_k(a_{ik})$ such that for all $v\in S$ and all $g\in \Gamma$
\be\label{79} \lim_{k\to\infty}\sum_i \bar
V(a_{ik})g(a_{ik}) |\epsilon_{ik} \O|= \int_\O \bar
V(x)g(x)\,\md x\ ,\ee
where
$$\bar V(x)=\int_{\reg} v(s)\nu_x(\md s)\ .$$ 
In view of  Lemma~\ref{homocase}, we see  that $\{\nu_{a_{ik}}\}_{x\in\O}\in\Grad$ is a homogeneous gradient Young measure and we  call $\{\nabla y^{ik}_j\}_{j\in\N}\subset W^ {1,\infty,-\infty}_+(\O;\R^2)$  its generating sequence. We know that we can consider  $\{y^{ik}_j\}_{j\in\N}\subset \mathcal{U}^{\tilde\varrho}_{\nabla u(a_{ik})}$ for arbitrary $+\infty >\tilde\varrho>\varrho$. Hence

\be\label{imp14} 
\lim_{j\to\infty} \int_\O v(\nabla y_j^{ik}(x))g(x)\,\md x=\bar V(a_{ik})\int_\O g(x)\,\md x\ 
\ee
and, in addition, $y^{ik}_j$ weakly$^*$ converges to the map $x \mapsto \nabla u(a_{ik})x$ for $j\to\infty$ in $W^ {1,\infty}(\O;\R^2)$ and due to the Arzela-Ascoli theorem also uniformly on $C(\bar\O;\R^2)$. 

Further, consider for $k\in\mathbb{N}$ $y_k\in W^{1,\infty}(a_{ik}+\epsilon_{ik}\O;\R^2)$ defined  for  $x\in a_{ik}+\epsilon_{ik}\O$  by
$$
y_k(x): = 
 u(a_{ik})+\epsilon_{ik}y^{ik}_j\left(\frac{x-a_{ik}}{\epsilon_{ik}}\right) 
$$
where $j=j(k, i)$ will be chosen later. Note that the above formula defines  $y_k$  almost everywhere in $\O$.
We write for almost every $x\in a_{ik}+\epsilon_{ik}\O$  that 
\begin{eqnarray}\label{upscale}
|u(x)-y_k(x)|&\leq&\left |u(x)-u(a_{ik})-\epsilon_{ik}\nabla u(a_{ik})\left(\frac{x-a_{ik}}{\epsilon_{ik}}\right)\right|\nonumber\\
&+&\epsilon_{ik}\left|\nabla u(a_{ik})\left(\frac{x-a_{ik}}{\epsilon_{ik}}\right)-y^{ik}_j\left(\frac{x-a_{ik}}{\epsilon_{ik}}\right)\right|\leq \frac{2\epsilon_{ik}}{k}\ ,
\end{eqnarray}
if $j$ is large enough. The first term on the right-hand side is bounded by $\epsilon_{ik}/k$ because of \eqref{derivative} while the second one due to the \ONE uniform \EEE convergence of $y^{ik}_j\to x \mapsto \nabla u(a_{ik})x$. Notice that $y_k$ as well as $u$ are bi-Lipschitz and orientation preserving on   $a_{ik}+\epsilon_{ik}\O$.
If $x\in a_{ik}+\epsilon_{ik}\O$ we set $\tilde x= (x-a_{ik})/\epsilon_{ik}\in\O$ and define $\tilde u(\tilde x)=\epsilon_{ik}^{-1}u(a_{ik}+\epsilon_{ik}\tilde x)$ and $\tilde y_k(\tilde x)= \epsilon_{ik}^{-1}y_k(a_{ik}+\epsilon_{ik}\tilde x)$ so that we get by \eqref{upscale} for all $x\in\O$
$$
|\tilde u(\tilde x)-\tilde y_k(\tilde x)|\le \frac2k\ 
$$
Additionally, note that  the bi-Lipschitz constant of $\tilde y_k$, $k\in\N$ is again $L$.

Hence, we can take $k>0$ large enough that $\|\tilde u-\tilde y_k\|_{C(\bar\O;\R^2)}$ is arbitrarily small. Therefore, we can use Theorem~\ref{corollary-cut-off}  and modify $\tilde y_k$ so that it has the same trace as  $\tilde u$ on the boundary of  $\O$.  Let us call this modification $\tilde u_k$, i.e., 
$$
\tilde u_k(\tilde x)=\begin{cases}
 \tilde y_k(\tilde x) &\mbox{ if $x\in\O$},\\
  \tilde u(\tilde x) & \mbox{otherwise}.
\end{cases}
$$ 
Then we proceed in the opposite way to define for $x=a_{ik}+\epsilon_{ik}\tilde x$:  $u_k(x)= \epsilon_{ik}\tilde u_k(\tilde x)$.

 Then, since $\{u_k\}_{k\in\N}$ is bounded in $W^{1,\infty}(\O;\R^2)$, we  may assume the weak$^*$ convergence of $u_k$ to $u$.  It remains to show that every $u_k$ is bi-Lipschitz. To do so, we again apply Theorem~\ref{jmball}. We see that for every $k\in\mathbb{N}$ $\det\nabla u_k>0$. Further, $\sup_{k\in\N}|(\nabla u_k)^{-1}|<+\infty$ follows from construction of the sequence, and $u_k=u$ on $\partial\O$, so that $u_k$ is indeed bi-Lipschitz. 

For $k,i$ fixed we take $j=j(k,i)$ so large that for all
$(g,v)\in \Gamma\times S$
$$
\left|\epsilon_{ik}^2\int_\O g(a_{ik}+\epsilon_{ik}y)v(\nabla
u_j^{ik}(y))\,\md y -\bar
V(a_{ik})\int_{a_{ik}+\epsilon_{ik}\O}g(x)\,\md
x\right|\le\frac{1}{2^ik}\ .$$
 Using
this estimate and (\ref{imp14}) we get for any $(g,v)\in
\Gamma\times S$
\begin{eqnarray*}
\lim_{k\to\infty}\int_\O g(x)v(\nabla u_k(x))\,\md x &=& \lim_{k\to\infty}\sum_i\epsilon_{ik}^n\int_\O g(a_{ik}+\epsilon_{ik}y)v(\nabla u_j^{ik}(y))\,\md y\\
&=& \lim_{k\to\infty}\sum_i\bar
V(a_{ik})\int_{a_{ik}+\epsilon_{ik}\O}g(x)\,\md x=
\int_\O \bar V(x)g(x)\,\md x \\
&=& \int_\O\int_{\R^{2\times 2}}
v(s)\nu_x(\md s)g(x)\, \md x \ .
\end{eqnarray*}

\hfill$\Box$

\bigskip

\subsection{Proofs of Corollary~\ref{wlsc} and Proposition~\ref{minimizer}}

{\it Proof of Corollary~\ref{wlsc}.}
For showing the weak lower semicontinuity, we realize that the sequence $\{\nabla y_k\}_{k \in \mathbb{N}}$ generates a measure in $\GradPlus$ and so if $v$ is bi-quasiconvex we easily have from \eqref{qc0}
$$
 \int_\Omega v(\nabla y(x)) \md x  = \int_\Omega Z v(\nabla y(x)) \md x \le\int_\Omega \int_{\regplus} v(s)\nu_x(\md s) \d x= \liminf_{k \to \infty} \int_\Omega v(\nabla y_k) \md x.
$$

On the other hand, we realize that  every $y\in W^{1,\infty,-\infty}_A(\O;\R^2)$ defines a  homogeneous Young measure $\nu\in\mathcal{GY}^{\infty,-\infty}_+(\O;\R^{2\times 2})$ by setting
$$\int_{\R^{2\times 2}}f(s)\nu(\d s)=|\O|^{-1}\int_\O f(\nabla y(x))\,\d x\ $$
for every $f$ continuous on matrices with positive determinant.

 Notice that the first moment of $\nu$ is $A$. Let $\{\nabla y_k\}_{k \in \mathbb{N}}$ be a generating sequence for $\nu$   which can be taken such that $\{y_k\}_{k \in \mathbb{N}} \subset  W^{1,\infty,-\infty}_A(\O;\R^2)$.  Moreover, the weak* limit of $\nabla y_k$ is $A$.  As we assume that $I(y)=\int_\O v(\nabla y(x))\,\d x$ and that $I$ is weakly$^*$ lower semicontinuous on  $W^{1,\infty,-\infty}_A(\O;\R^2)$ we get

$$|\O|v(A)\le \liminf_{k\to\infty}I(y_k)=\int_\Omega \int_{\R^{2\times 2}}v(s)\nu(\d s) \d x = \int_\O v(\nabla y(x))\,\d x\ ,$$
which shows that $v$ is bi-quasiconvex.
\hfill$\Box$

\bigskip

{\it Proof of Proposition~\ref{minimizer}.}
Notice that $u_0\in\mathcal{A}$ so that the admissible set is nonempty. Let $\{u_k\}_{k\in\N}\subset\mathcal{A}$ be a minimizing sequence for $I_\varepsilon$, i.e., 
$\lim_{k\to\infty} I_\varepsilon(u_k)=\inf_\mathcal{A}I_\varepsilon\ge 0$.  
Hence, $\|\nabla u\|_{L^\infty(\O;\R^{2\times 2})}\le C$ and $\|\nabla u^{-1}\|_{L^\infty(u_0(\O);\R^{2\times 2})}\le C$ for some finite $C>0$. Applying a Poincar\'{e} inequality we get that $\{u_k\}$ is  bounded in $W^{1,\infty,-\infty}_+(\O;\R^2)$. Therefore, there is a  subsequence converging weakly* to some $u\in W^{1,\infty,-\infty}_+(\O;\R^2)$.
Compactness of the trace operator ensures that $u=u_0$ on the boundary of $\O$. Consequently, $u\in\mathcal{A}$ and weak* lower semicontinuity of $I_\varepsilon$ finishes the argument. Indeed, 
as $v$ is bi-quasiconvex  the weak* lower semicontinuity of the first two terms is obvious. The last term is weak* lower semicontinuous  in view of Remark~\ref{convergence}.
\hfill $\Box$

\section{Cut-off technique preserving the bi-Lipschitz property}
\label{sect-cutOff}
One of the main steps in the characterization of gradient Young measures \cite{k-p1,pedregal} is to show that having a bounded sequence $\{y_k\}_{k\in\N}\subset W^{1,\infty}(\O;\R^2)$, such that it converges weakly$^*$ to $y(x):\O\mapsto\R^2$  and $\{\nabla y_k\}$ generates  a Young measure $\nu$, then there is a modified sequence $\{u_k\}_{k\in\N}\subset W^{1,\infty}(\O;\R^2)$, $u_k(x)=y(x)$  for $x\in\partial\O$ and $\{\nabla u_k\}$ still generates $\nu$.  Standard proofs of this fact use a cut-off technique based on convex combinations near the boundary; due to the non-convexity of our constraints, however, this could destroy the bi-Lipschitz property, so it is not at all suitable for our purposes. Therefore, we resort to a different approach borrowing from recent results by S.~Daneri and A.~Pratelli \cite{daneri-pratelli-extension,daneri-pratelli}. More precisely, the following theorem is a main ingredient of our approach.

 \begin{theorem}\label{cut-off}
 
 Let $\O\subset\R^2$ be a bounded Lipschitz domain, let $\mathrm{diam}\ \Omega>>\delta > 0$  and $L\ge 1$ be fixed.  Then there exists $\varepsilon > 0$ that is only dependent on $\delta$ and $L$ such that if $\tilde y, y \in W^{1,\infty, -\infty}_+(\O;\R^2)$ are L-bi-Lipschitz maps satisfying 
 $$
 \|\tilde y-y\|_{C(\overline{\Omega};\R^2)} \leq \varepsilon(\delta, L),
 $$
then  we can find another $\bar{c}(L)$-bi-Lipschitz map $u \in  W^{1,\infty, -\infty}_+(\O;\R^2)$ satisfying $u=y$ on $\partial\O$ and 
  $|\{x\in\O;\, \nabla u(x)\ne\nabla \tilde y(x)\}| \leq \delta$. 
 \end{theorem}

 The  following corollary allows us to modify convergent sequences at the boundary of $\O$.

\begin{corollary}\label{corollary-cut-off}
Assume that $\{y_k\}_{k\in\N}\subset W^{1,\infty,-\infty}_+(\O;\R^2)$ is a sequence of $L$-bi-Lipschitz maps and $y_k\stackrel{*}{\rightharpoonup} y$ in $W^{1,\infty,-\infty}_+(\O;\R^2)$ as $k\to\infty$. Then there is a  subsequence of $\{y_{k_n}\}_{n\in\N}$ and  $\{u_{k_n}\}_{n\in\N}\subset W^{1,\infty,-\infty}_+(\O;\R^2)$ bounded  such that 
$u_{k_n}\stackrel{*}{\rightharpoonup} y$  in $W^{1,\infty,-\infty}_+(\O;\R^2)$ as $n\to\infty$,  for all $n\in\N$ $u_{k_n}=y$ on $\partial\O$ and 
$\lim_{n\to\infty}|\{x\in\O;\, \nabla u_{k_n}\ne\nabla y_{k_n}\}|\to 0$. In particular, the sequences $\{\nabla y_{k_n}\}$ and $\{\nabla u_{k_n}\}$ generate the same Young measure.

\end{corollary}

\bigskip

\begin{proof}
Let   $\{\delta_n\}_{n\in\N}$ be a sequence of positive numbers converging to zero as $n\to\infty$. We  apply Theorem~\ref{cut-off} and uniform convergence of $\{y_k\}_{k\in\N}$ to $y$ in $C(\bar\O;\R^2)$  to find  $\{\varepsilon_n(\delta_n,L)\}_{n\in\N}$  and $\{y_{k_n}\}_{n\in\N}$ such that  $\|y_{k_n}-y\|_{C(\bar\O;\R^2)}\le \varepsilon_n(\delta_n,L)$. Use Theorem~\ref{cut-off}  with $\tilde y:=y_{k_n}$  to obtain  $u_{k_n}\in W^{1,\infty,-\infty}_+(\O;\R^2)$ with the mentioned properties.   
\end{proof}

{\it Proof of Thm.~\ref{cut-off}}
{
We devote the rest of this section to proving Theorem \ref{cut-off},  large parts of the proof, collected in its third section, are rather technical. Therefore, we start with an overview of the proof:

\vspace{2ex}

\noindent \textit{\textbf{Section 1 of the proof: Overview}}\\
We define the open set
$$
\Omega^\delta = \big\{x \in \Omega: \mathrm{dist}(x,\partial\O) < \delta \big\};
$$ 
now, we find $r = r(\delta)$ and a corresponding, suitable \emph{$r(\delta)$-tiling of $\Omega^\delta$}, i.e.\ a \emph{finite} collection of closed squares 
\begin{equation}
\Omega_r = \bigcup_{i=1}^N \mathcal{D}(z_i, r) \qquad \quad \text{with $z_i \in \Omega^\delta$} 
\label{tiling}
\end{equation}
that satisfies that $\Omega_r \subsetneq \Omega^\delta$ and that two squares have in common \emph{only} either a whole edge or a vertex. Furthermore, we require the tiling to be fine enough so that there exists a \emph{collection of edges $\Gamma$} satisfying the following properties:
\begin{itemize}
\item every continuous path connecting two points $x_1$ and $x_2$ such that $x_1 \in \partial\Omega$ and $x_2 \in \partial \Omega^\delta \setminus \partial \Omega$ crosses $\Gamma$,
\item $\Gamma \subset \mathrm{int} \, \Omega_r$.
\end{itemize}}
 This setting is best imagined in the case when $\Omega$ is simply connected. Then, $\Omega_r$ forms a thin strip of squares near the boundary and $\Gamma$ is a closed curve consisting of edges \emph{in the interior} of this strip. We will refer to the special case of a simple connected domain for a better imagination of the introduced concepts at several places bellow; nevertheless, simple connectivity of $\Omega$ is never explicitly used and, in fact, not needed. 

Further, we separate $\Omega$ into three parts:
$$
\Omega = \Omega_\mathrm{bulk} \cup \Omega_r \cup \Omega_\mathrm{bound},
$$
where 
\begin{align*}
\Omega_\mathrm{bulk} &= \{x \in \Omega \setminus \Omega_r: \text{every continuous path from $x$ to $\partial \Omega$ crosses $\Omega_r$}.\} \\
\Omega_\mathrm{bound} &= \Omega^\delta\setminus(\Omega_\mathrm{bulk} \cup \Omega_r).
\end{align*}
Let us again, for a moment, think of a simply connected $\Omega$. Then, $\Omega_\mathrm{bulk}$ forms the interior of the domain, $\Omega_r$ is the thin strip of squares and $\Omega_\mathrm{bound}$ is also a strip that reaches up to $\partial \Omega$ and is not tiled.

With these basic notations set, we explain how we construct the cut-off.  Let us choose $\varepsilon = \frac{r(\delta)}{12L^3}$ so that we have that
\begin{equation}
\|\tilde y-y\|_{C(\overline{\Omega};\R^2)} \leq \frac{r(\delta)}{12L^3}.
\label{LargeK}
\end{equation}
Now, we alter $\tilde y$ on $\Omega_r$ to obtain the function $u_\delta: \Omega_r \to \mathbb{R}^2$ that has the property that $[u_\delta]_{\mid_{\partial \Omega_r\cap \partial \Omega_\mathrm{bulk}}} = \tilde y$ and $[u_\delta]_{\mid_{\partial \Omega_r\cap \partial \Omega_\mathrm{bound}}}=y$. If we think once more of simple connected $\Omega$, this means that on the inner boundary of $\Omega_r$ we obtain the function $\tilde y$ while on the outer boundary we already have the sought boundary data. 

We will give a precise definition of $u_\delta$ in the next section of the proof. In fact, in view of the available extension Theorem \ref{extensionTheorem}, it is sufficient to give a definition of $u_\delta$ on all the edges in $\Omega_r$, which we will exploit. Namely, on the edges the ``fitting'' of $\tilde y$ to $y$ is essentially one-dimensional and hence our technique will be essentially a linear interpolation. 

 In the third section of the proof, which is the most technical one, we then show that $u_\delta$, thus so far defined only on the edges, is $18L$-bi-Lipschitz (cf. \eqref{biL}) and so extending it to $\Omega_r$ via Theorem \ref{extensionTheorem} will yield a $\bar{c}(L)$-bi-Lipschitz function $u_\delta: \Omega_r \to \mathbb{R}^2$ having the above described properties. \ONE Indeed, $\partial u_\delta(\mathcal{D}(z_i,r))= u_\delta(\partial \mathcal{D}(z_i,r))$ for all admissible $i$, so that $u_\delta:\O_r\to\R^2$ is injective.  \EEE 

Therefore, we may define 
$$
u(x) = \begin{cases}
\tilde y(x) & \text{on $\Omega_\mathrm{bulk},$} \\
u_\delta(x) & \text{on $\Omega_r$} ,\\
y(x) & \text{on $\Omega_\mathrm{bound}$,}
\end{cases}
$$
It is obvious that the obtained mapping is Lipschitz and satisfies \WE $|\nabla u(x)^{-1}| > c(L)$ \EEE  a.e.~on $\Omega$. The injectivity  of \WE $u$ \EEE follows from the fact that $u(\Omega_\mathrm{bulk}), u(\Omega_\mathrm{bound})$ and $u(\Omega_r)$ are mutually disjoint, which is a consequence of the ``fitting'' boundary data through $[u_\delta]_{\mid_{\partial \Omega_r\cap \partial \Omega_\mathrm{bulk}}} = \tilde y$ and $[u_\delta]_{\mid_{\partial \Omega_r\cap \partial \Omega_\mathrm{bound}}}=y$. Thus, the mapping $u$ is globally bi-Lipschitz and hence orientation preserving since it preserves orientation on $\Omega_\mathrm{bulk}$.

\begin{figure}[h!]
\centering
\subfigure[r-tiling of the set $\Omega^\delta$]{\includegraphics[width = 0.35\paperwidth]{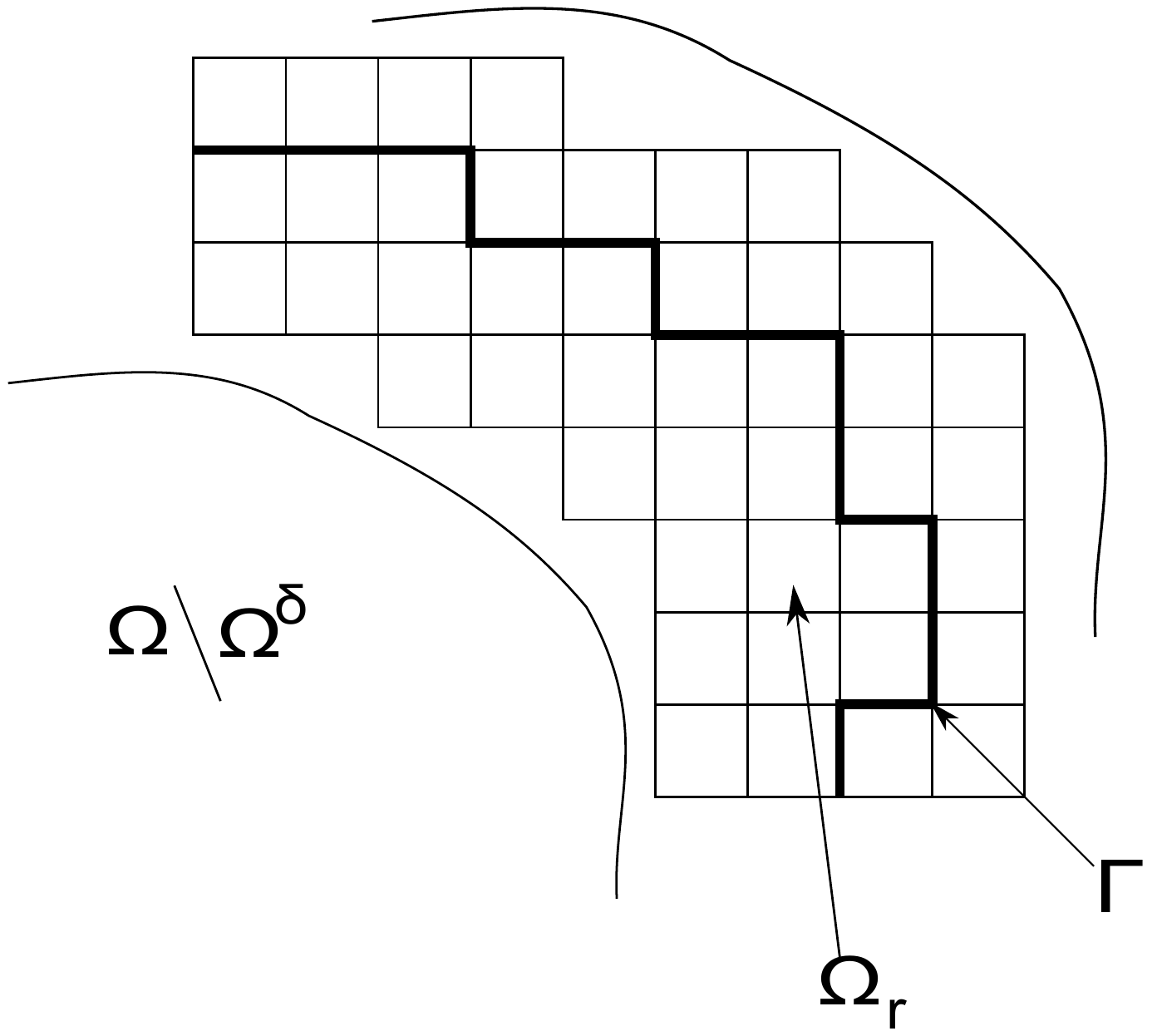}}
\subfigure[Detail of cross on $\Gamma$]{\includegraphics[width = 0.35\paperwidth]{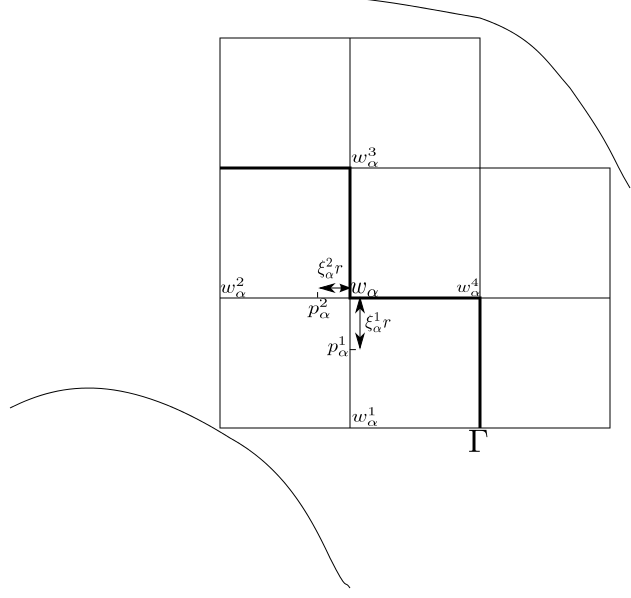}} 
\caption{Tiling near boundary and detail of one cross}
\label{Fig_CutOFF}
\end{figure}

\vspace{2ex}

\noindent \textit{\textbf{Section 2 of the proof: Partitioning of the grid and definition of $u_\delta$}}\\
In this section we give a precise definition of $u_\delta(x)$ on the  \emph{grid} of the tiling $\Omega_r$, denoted $\mathcal{Q}$, which consists of all edges of $\Omega_r$; in other words,
$$
\mathcal{Q} = \bigcup_{i=1}^N \partial \mathcal{D}(z_i, r) \qquad \quad \text{with $z_i$ as in \eqref{tiling}}. 
$$
Clearly, $\Gamma \subset \mathcal{Q}$ and we divide $\mathcal{Q}$ into two other parts 
$$
\mathcal{Q} = \mathcal{Q}^\mathrm{outer} \cup \Gamma \cup \mathcal{Q}^\mathrm{inner},
$$
defined through

\begin{align}
\mathcal{Q}^\mathrm{inner} &=\big \{x \in \mathcal{Q}\setminus \Gamma; \text{every continuous path connecting $x$ to $\partial \Omega$ crosses $\Gamma$}\}, \\
\mathcal{Q}^\mathrm{outer} &= \mathcal{Q}\setminus(\Gamma \cup \mathcal{Q}^\mathrm{inner}).
\end{align}
 The names of these two other parts are lent from the situation when $\Omega$ is simply connected;  namely, then $\mathcal{Q}^\mathrm{inner}$ corresponds to those edges that are ``further away'' from the boundary than $\Gamma$ and so in the ``interior'' while $\mathcal{Q}^\mathrm{outer}$ are the edges in the exterior. Nevertheless, as already stressed above, simple-connectivity of $\Omega$ is not needed.

For further convenience, we shall fix some notation (in accord with \cite{daneri-pratelli}); see also Figure \ref{Fig_CutOFF}(b).  We shall denote 
\begin{itemize}
\item $w_\alpha$ any vertex of the grid $\mathcal{Q}$ that lies on  $\Gamma$,
\item for any $w_\alpha$ we denote $w_\alpha^i$ all vertices that are at distance of $r$ to $w_\alpha$; note that from construction there always exist 4 such vertices (as $w_\alpha$ cannot lie on the boundary of $\Omega_r$),
\item for any $w_\alpha$ the  largest numbers $\xi_\alpha^i > 0$ that satisfy 
\begin{align*}
\big|\tilde y\big(w_\alpha+\xi_\alpha^i(w_\alpha^i - w_\alpha)\big)- y(w_\alpha) \big| &= \frac{r}{4L} & &\text{if the edge $w_\alpha w_\alpha^i \subset \mathcal{Q}^\mathrm{inner}$}, \\
\big|y\big(w_\alpha+\xi_\alpha^i(w_\alpha^i - w_\alpha)\big)- y(w_\alpha) \big| &= \frac{r}{4L} & &\text{else};
\end{align*}

\item we call the ``boundary cross'' the set 
$$
Z_\alpha = \bigcup_{i=1}^{4} \big\{w_\alpha+ t (w_\alpha^i-w_\alpha) : 0 \leq t \leq \xi_\alpha^i \big\}
$$
and denote the extremals of this cross $p_\alpha^1 \ldots p_\alpha^4$.
\end{itemize}

 It is due to the $L$-bi-Lipschitz property of $\tilde y$ and $y$ as well as \eqref{LargeK} that all the concepts above are well defined. In particular, we can assure that
\begin{equation}
\text{the numbers $\xi_\alpha^i$ can be found in the interval $[1/(6L^2),1/3]$,} \label{numbers_xi} 
\end{equation}
so that the boundary crosses are mutually disjoint. We postpone the proof of \eqref{numbers_xi} until the end of this section.

Now, we are in the position to define the sequence $u_{k\delta}(x)$ on $\mathcal{Q}$ as follows: first, we define $u_{\delta}(x)$ everywhere in $\mathcal{Q}$ except for the boundary crosses:
$$
u_\delta(x) = \begin{cases} \tilde y(x) & \text{if $x \in \mathcal{Q}^\mathrm{inner}\setminus (\bigcup_\alpha Z_\alpha),$} \\
y(x) & \text{if $x \in (\mathcal{Q}^\mathrm{outer}\cup \Gamma)\setminus (\bigcup_\alpha Z_\alpha);$}
\end{cases}
$$
while on the cross the $u_\delta$ will be continuous and piecewise affine, i.e. 
$$
u_\delta(w_\alpha + t(w_\alpha^i- w_\alpha)) = \begin{cases}
y(w_\alpha) + \frac{t}{\xi_\alpha^i }\Big(\tilde y(p_\alpha^i)- y(w_\alpha) \Big) & \text{if $w_\alpha w_\alpha^i \subset \mathcal{Q}^\mathrm{inner}$ and $t \in [0,\xi_\alpha^i ]$,} \\
y(w_\alpha) + \frac{t}{\xi_\alpha^i }\Big(y(p_\alpha^i)- y(w_\alpha) \Big) & \text{if $w_\alpha w_\alpha^i \not \subset \mathcal{Q}^\mathrm{inner}$ and $t \in [0,\xi_\alpha^i ].$} 
\end{cases}
$$

The rough idea behind this construction is that the matching, or the cut-off, actually happens on the boundary crosses where we, on each edge, replace $\tilde y$ as well as $y$ by an affine map. By adjusting the slopes of these affine replacements we get a continuous piecewise affine, and hence bi-Lipschitz, map on the cross. What we need to show are then, essentially, the following two properties of such a replacement: it connects in a bi-Lipschitz way to $u_{\delta}$ along the endpoints of the boundary cross and the adjustment of the slopes needed to obtain continuity is just small so that the overall $L$-bi-Lipschitz property is not affected much.

For the former, we mimic the strategy of S.~Daneri and A.~Pratelli \cite{daneri-pratelli} who were also able to connect an affine replacement of a bi-Lipschitz function to the original map. The latter is due to the fact that $\tilde y$ and $y$ are suitably close to each other (as expressed by the property \eqref{LargeK}) which assures that the change of slope on the cross needed for the cut-off will depend just on $L$.

We will show in the next section that $u_\delta$ is $18L$-bi-Lipschitz on $\mathcal{Q}$; cf.\ \eqref{biL}. Therefore, we can apply Theorem~\ref{extensionTheorem} to extend $u_\delta$ from $\mathcal{Q}$ (without changing the notation) to each 
square of the tiling. As for every square $\mathcal{D}(z_i,r)$ of the tiling we have that $\partial u_\delta(\mathcal{D}(z_i,r))= u_\delta(\partial \mathcal{D}(z_i,r))$ we see that the extended mapping is globally injective on $\O_r$. 

\vspace{1ex}

\noindent \emph{Proof of \eqref{numbers_xi}:}

For $w_\alpha w_\alpha^i \subset \mathcal{Q}^\mathrm{inner}$, we notice that the function $t \mapsto \big|\tilde y\big(w_\alpha+t(w_\alpha^i - w_\alpha)\big)- y(w_\alpha) \big| $ is continuous on $[0,1]$ and, owing to \eqref{LargeK}, smaller or equal than $\frac{r}{12L^3}$ in $0$ while in $t=1$ we have that 
$$
\Big|\tilde y(w_\alpha^i) - \tilde y(w_\alpha)+ \tilde y(w_\alpha)- y(w_\alpha) \Big| \geq \Big|\frac{r}{L}-\frac{r}{12L^3} \Big| \geq \frac{r}{4L};
$$
which yields the existence of  $\xi_\alpha^i \in [0,1]$ such that  
$$
\big|\tilde y\big(w_\alpha+t(w_\alpha^i - w_\alpha)\big)- y(w_\alpha) \big| = \frac{r}{4L}.
$$

To establish the bounds on $\xi_\alpha^i$, we note that
\begin{align*}
\frac{r}{4L}&=\Big|\tilde y\big(w_\alpha+\xi_\alpha^i(w_\alpha^i - w_\alpha)\big)- y(w_\alpha) \Big| =
\Big|\tilde y\big(w_\alpha+\xi_\alpha^i(w_\alpha^i - w_\alpha)\big) + \tilde y(w_\alpha)- \tilde y(w_\alpha) - y(w_\alpha) \Big| \\ &\leq L\xi_\alpha^i r + \frac{r}{12L^3} \leq 
L\xi_\alpha^i r + \frac{r}{12L},
\end{align*}
i.e.\ $\xi_\alpha^i \geq 1/(6L^2)$. On the other hand we have that 
\begin{align*}
\frac{r}{4L}&=\Big|\tilde y\big(w_\alpha+\xi_\alpha^i(w_\alpha^i - w_\alpha)\big)- y(w_\alpha) \Big| =
\Big|\tilde y\big(w_\alpha+\xi_\alpha^i(w_\alpha^i - w_\alpha)\big) + \tilde y(w_\alpha)- \tilde y(w_\alpha) - y(w_\alpha) \Big| \\& \geq \frac{r}{L}\big(\xi_\alpha^i - \frac{1}{12L^2}\big) \geq \frac{r}{L}\big(\xi_\alpha^i - \frac{1}{12}\big),
\end{align*}
which is satisfied if $0\leq \xi_\alpha^i \leq 1/3$.

In the case when $w_\alpha w_\alpha^i  \not\subset \mathcal{Q}^\mathrm{inner}$, we proceed in a similar way and rely just on the bi-Lipschitz property of $y$; exploiting \eqref{LargeK} is not necessary.

\vspace{2ex}

\noindent \textit{\textbf{Section 3 of the proof: Bi-Lipschitz property of $u_\delta$:}}\\
The function $u_\delta$ defined in the previous section is continuous on the grid $\mathcal{Q}$ and we claim that it is even  bi-Lipschitz, i.e.\ (as long as \eqref{LargeK} holds true)
\begin{equation}
18 L |z-z'| \geq |u_\delta(z)-u_\delta(z')| \geq \frac{1}{18L}|z-z'| \qquad \quad \forall z,z' \in \mathcal{Q}
\label{biL}
\end{equation}

The proof of this claim is the content of this section and will be performed in several steps.

\noindent \emph{Step 1 of the proof of \eqref{biL}: Suppose that $z$ and $z'$ lie in  $Z_\alpha$.} \\
Let us first consider the situation when both $z, z'$ lie on the same edge; i.e.\ $z,z' \in w_\alpha w_\alpha^i$ for some $i=1\ldots4$. In this a case $u_\delta$ is affine and we have that
\begin{align*}
\frac{|u_\delta(z)-u_\delta(z')|}{|z-z'|} &= \frac{\big|u_\delta(w_\alpha) - u_\delta(p_\alpha^i)\big|}{\xi_\alpha^i r} \\&\!\!\!\!\!\!\!\!=
\begin{cases}
\frac{\big|\tilde y(p_\alpha^i) - \tilde y(w_\alpha) + \tilde y(w_\alpha)-y(w_\alpha)\big|}{\xi_\alpha^i r} \geq \frac{1}{L}  -  \frac{1}{\xi_\alpha^i r} \frac{r}{12L^3} \geq \frac{1}{L} -  \frac{6 L^2}{r} \frac{r}{12L^3} \geq \frac{1}{2L} & \!\!\!\!\text{if $w_\alpha w_\alpha^i \subset \mathcal{Q}^\mathrm{inner}$}\\
\frac{\big|y(p_\alpha^i)-y(w_\alpha)\big|}{\xi_\alpha^i r} \geq \frac{1}{L} \geq \frac{1}{2L} &\!\!\!\! \text{if $w_\alpha w_\alpha^i \not \subset \mathcal{Q}^\mathrm{inner}$.}
\end{cases}
\end{align*}
Similarly,
\begin{align*}
\frac{|u_\delta(z)-u_\delta(z')|}{|z-z'|} &= \frac{\big|u_\delta(w_\alpha) - u_\delta(p_\alpha^i)\big|}{\xi_\alpha^i r} \\ 
&\!\!\!\!\!\!\!\!= \begin{cases}
\frac{\big|\tilde y(p_\alpha^i) - \tilde y(w_\alpha) + \tilde y(w_\alpha)-y(w_\alpha)\big|}{\xi_\alpha^i r} \leq L  +  \frac{1}{\xi_\alpha^i r} \frac{r}{12L^3} \leq L +  \frac{6 L^2}{r} \frac{r}{12L^3} \leq 2L & \text{if $w_\alpha w_\alpha^i \subset \mathcal{Q}^\mathrm{inner}$}\\
\frac{\big|y(p_\alpha^i)-y(w_\alpha)\big|}{\xi_\alpha^i r} \leq L \leq 2L & \text{if $w_\alpha w_\alpha^i \not \subset \mathcal{Q}^\mathrm{inner}$.}
\end{cases}
\end{align*}

If $z$ and $z'$ are not on the same edge let, for example, $z \in w_\alpha p_\alpha^1$ and $z' \in w_\alpha p_\alpha^2$. Moreover, we may assume, without loss of generality, that
$$
|u_\delta(z)-y(w_\alpha)| \leq |u_\delta(z') - y(w_\alpha)|
$$ and, hence,  define $z''$ in the segment $w_\alpha z'$ such that 
$$
|u_\delta(z)-y(w_\alpha)|=|u_\delta(z'')-y(w_\alpha)|.
$$
Then, as the points $u_\delta(z)$, $u_\delta(z'')$ and $u_\delta(z')$ form a triangle that is obtuse at $u_\delta(z'')$ (cf. also Figure \ref{obtuse}) we may apply Remark \ref{ObtuseTriangle} to obtain
\begin{align}
|u_\delta(z) - u_\delta(z')| &\geq \frac{1}{\sqrt{2}}\Big(|u_\delta(z) - u_\delta(z'')| + |u_\delta(z') - u_\delta(z'')| \Big) \nonumber\\ &\geq 
\frac{1}{\sqrt{2}}\Big(|u_\delta(z) - u_\delta(z'')| + \frac{1}{2L}|z'-z''| \Big)
\label{Step1_1}
\end{align}
since the points $z'$, $z''$ lie on the same edge where we already proved the bi-Lipschitz property. Further, by the fact that $u_\delta$ is piecewise affine on the cross, \footnote{Notice that on any the segment $w_\alpha p_\alpha^i$ we can write $u_\delta(t) = u_\delta(w_\alpha) + t (u_\delta(p_\alpha^i) -u_\delta(w_\alpha))$. Therefore, the points $z$,$z''$ correspond to such $t$, $t''$ that $t |u_\delta(p_\alpha^1) -u_\delta(w_\alpha)| = t''|u_\delta(p_\alpha^2) -u_\delta(w_\alpha)|$. By definition, however, $|u_\delta(p_\alpha^1) -u_\delta(w_\alpha)| = |u_\delta(p_\alpha^2) -u_\delta(w_\alpha)| = \frac{r}{4L}$ so that $t=t''$.}
\begin{figure}[h!]
\centering
\includegraphics[width = 0.5\paperwidth]{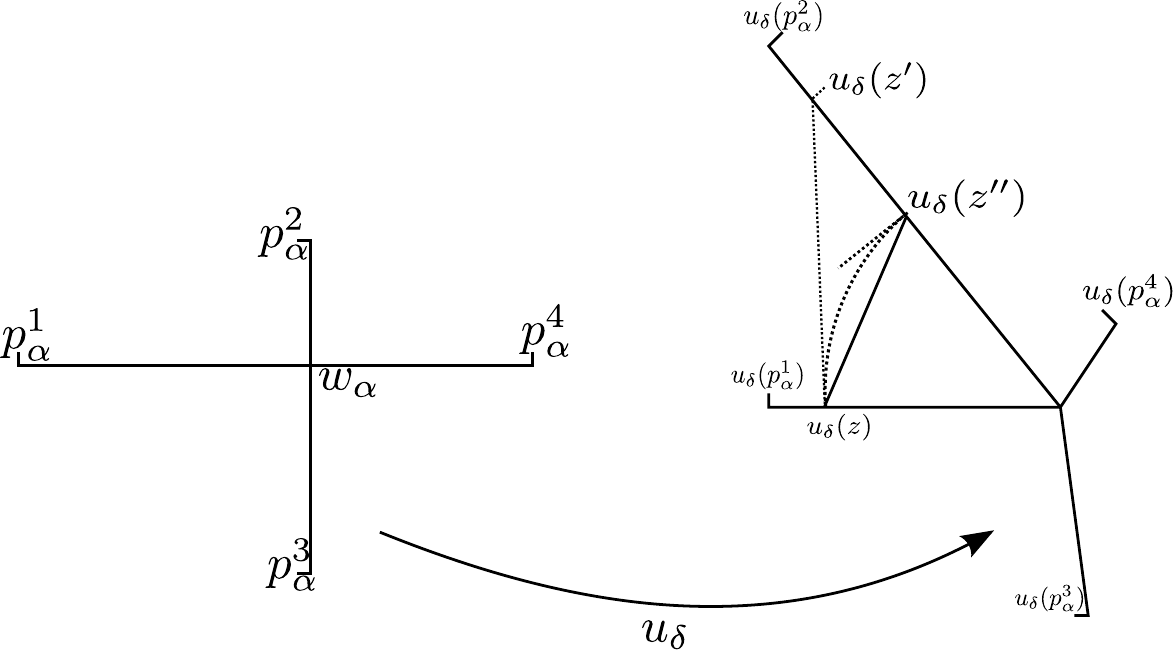}
\caption{The obtuse triangle formed by  $u_\delta(z)$, $u_\delta(z'')$ and $u_\delta(z')$ in the image of the boundary cross as needed in Step 1. Notice that since $u_\delta$ is piecewise affine on the cross, each segment of the cross forms again a part of a straight line.}
\label{obtuse}
\end{figure}
\begin{align*}
\frac{|u_\delta(z) - u_\delta(z'')|}{|z-z''|} &= \frac{|u_\delta(p_\alpha^1) - u_\delta(p_\alpha^2)|}{|p_\alpha^1-p_\alpha^2|} \\ & \!\!\!\!\!\!\!\!\!\!\!\!\!\!\!\!= 
\begin{cases}
\frac{|\tilde y(p_\alpha^1) - \tilde y(p_\alpha^2)|}{|p_\alpha^1-p_\alpha^2|} \geq \frac{1}{L} & \text{if both $p_\alpha^1$, $p_\alpha^2$ lie in $\mathcal{Q}^\mathrm{inner}$} \\
\frac{|y(p_\alpha^1) - y(p_\alpha^2)|}{|p_\alpha^1-p_\alpha^2|} \geq \frac{1}{L} & \text{if neither $p_\alpha^1$ nor $p_\alpha^2$ lies in $\mathcal{Q}^\mathrm{inner}$}  \\
\frac{|\tilde y(p_\alpha^1) - \tilde y(p_\alpha^2) + \tilde y(p_\alpha^2) - y(p_\alpha^2)|}{|p_\alpha^1-p_\alpha^2|} \geq \frac{1}{L} - \frac{1}{|p_\alpha^1-p_\alpha^2|}\frac{r}{12L^3} \geq \frac{1}{2L}, & \text{$p_\alpha^1 \in \mathcal{Q}^\mathrm{inner}$ $p_\alpha^2, \notin \mathcal{Q}^\mathrm{inner}$}
\end{cases}
\end{align*}
where we realized that $|p_\alpha^1-p_\alpha^2| \geq \frac{r}{6L^2}$ because the triangle formed by the points $p_\alpha^1$,$w_\alpha$, $p_\alpha^2$ is right angled or a line. Notice also that the situation when  $p_\alpha^1 \notin \mathcal{Q}^\mathrm{inner}$, $p_\alpha^2 \in \mathcal{Q}^\mathrm{inner}$ is completely symmetrical to the already covered case.
So, returning to \eqref{Step1_1}, we have by the triangle inequality 
$$
|u_\delta(z) - u_\delta(z')| \geq \frac{\sqrt{2}}{4L} |z-z'|\ .
$$

On the other hand, by exploiting that the triangle formed by the points $z$,$z'$ and $w_\alpha$ is either right angled or a line, we get that
$$
|u_\delta(z) - u_\delta(z')| \leq |u_\delta(z) - y(w_\alpha) + y(w_\alpha) - u_\delta(z')| \leq 2L\big(|z-w_\alpha|+|z'-w_\alpha|\big) \leq 2L \sqrt{2} |z-z'|.
$$

\vspace{2ex}

\noindent \emph{Step 2 of the proof of \eqref{biL}: Suppose that $z \notin Z_\alpha$ and $z' \notin Z_\beta$ for all $\alpha, \beta$.} \\
Notice that we only have to investigate the case when $z \in \mathcal{Q}^\mathrm{inner}$ and $z' \notin \mathcal{Q}^\mathrm{inner}$ for the other options are trivial. Then, however, we have that $|z-z'| \geq \frac{r}{6L^2}$ and so the Lipschitz property follows immediately as 
$$
|u_\delta(z)-u_\delta(z')| \leq |y(z)-\tilde y(z)+\tilde y(z)-\tilde y(z')| \leq \frac{r}{6L^2}\,\frac{1}{2L} + L|z-z'| \leq 2L|z-z'|.
$$
On the other hand,
$$
\frac{|u_\delta(z)-u_\delta(z')|}{|z-z'|} = \frac{|y(z)-\tilde y(z)+\tilde y(z)-\tilde y(z')|}{|z-z'|} \geq \frac{1}{L} - \frac{r}{12L^2 |z-z'|} \geq \frac{1}{2L} 
$$\vspace{2ex}

\noindent \emph{Step 3 of the proof of \eqref{biL}: Suppose that $z \in Z_\alpha$ and $z' \notin Z_\beta$ for all $\beta$.} \\
To obtain the lower bound in \eqref{biL} we rely on Remark \ref{BallArgument}; indeed the choice of $z$, $z'$ is such that  $u_\delta(z')$ lies outside the ball $\mathcal{B}(y(w_\alpha); \frac{r}{4L})$ while $u_\delta(z) \in \mathcal{B}(y(w_\alpha); \frac{r}{4L})$. In particular, we may assume that $u_\delta(z) $ lies on the segment $y(w_\alpha) u_\delta(p_\alpha^1)$ (recall that $u_\delta$ is affine on the cross). So, 
$$
|u_\delta(z) - u_\delta(z')| \geq \frac{|u_\delta(p_\alpha^1)-u_\delta(z)| + |u_\delta(p_\alpha^1)-u_\delta(z')|}{3} \ .
$$ Clearly, we only have to care about the latter term on the right hand side. Employing  \eqref{LargeK} and the triangle inequality, we get that 
\begin{align*}
\frac{|u_\delta(p_\alpha^1) - u_\delta(z')|}{|p_\alpha^1-z'|} \geq 
\begin{cases}
\frac{|\tilde y(p_\alpha^1) - \tilde y(z')|}{|p_\alpha^1-z'|} \geq \frac{1}{L} & \text{if $p_\alpha^1, z' \in \mathcal{Q}^\mathrm{inner}$} \\
\frac{|y(p_\alpha^1) - y(z')|}{|p_\alpha^1-z'|} \geq \frac{1}{L} & \text{if $p_\alpha^1, z' \notin \mathcal{Q}^\mathrm{inner}$} \\ 
\frac{|\tilde y(p_\alpha^1) - \tilde y(z') + \tilde y(z') - y(z')|}{|p_\alpha^1-z'|} \geq \frac{1}{L}-\frac{6L^2}{12L^3} \geq \frac{1}{2L} & \text{if $p_\alpha^1 \in \mathcal{Q}^\mathrm{inner}$ and $z' \notin \mathcal{Q}^\mathrm{inner}$;}
\end{cases}
\end{align*}
where, in the last case, $p_\alpha^1$ and $z'$ necessarily lie in different edges and so $|p_\alpha^1-z'| \geq \frac{r}{6L^2}$. Notice that since the r\^{o}le of $p_\alpha^1$ and $z'$ is symmetric we really exhausted all possibilities belonging to this step. Summing up,
$$
|u_\delta(z) - u_\delta(z')| \geq \frac{|z-z'|}{6L}.
$$

To obtain the upper bound,  we first realize that if $z' $ is at the boundary to the cross, i.e.\@ $z' = p_\alpha^i$ for some $i=1\ldots4$, the procedure from Step 2 applies in verbatim. Therefore, we may restrict our attention to the situation in which $z' $ is strictly in the interior of the cross; then,  since 
all $p_\alpha^i$ are at distance at most $r/3$ from $w_\alpha$ and since $z' \notin w_\alpha p_\alpha^i$ $\forall i$, at least one of these $p_\alpha^i$ has to satisfy that the triangle $z, p_\alpha^i, z'$ has an obtuse (or right) angle at $p_\alpha^i$ (see Figure \ref{triangle}) -- let it for notational convenience be $ p_\alpha^1$. So, we are in the position to apply Remark \ref{ObtuseTriangle} below and estimate 
\begin{align*}
&|u_\delta(z) - u_\delta(z')|=|u_\delta(z) - u_\delta(p_\alpha^1) +  u_\delta(p_\alpha^1) - u_\delta(z')| \\
&\qquad \leq 
\left \{
\begin{array}{ll}
|u_\delta(z) - \tilde y(p_\alpha^1) +  \tilde y(p_\alpha^1) - \tilde y(z')| & \\
\quad \leq 2\sqrt{2} L \big(|z-p_\alpha^1|+|p_\alpha^1-z'|\big) \leq 4L|z-z'|\big) & \text{if $z' \in \mathcal{Q}^\mathrm{inner}$ and $p_\alpha^1 \in \mathcal{Q}^\mathrm{inner}$} \\
|u_\delta(z) - y(p_\alpha^1) + \tilde y(p_\alpha^1) - \tilde y(z') + y(p_\alpha^1) - \tilde y(p_\alpha^1)|  & \\
\quad \leq 2 \sqrt{2} L  \big(|z - p_\alpha^1| + |p_\alpha^1 - z'|  \big) + \frac{r}{6L^2}\, \frac{1}{2L} \leq 5L|z-z'| & \text{if $z' \in \mathcal{Q}^\mathrm{inner}$ and $p_\alpha^1 \notin \mathcal{Q}^\mathrm{inner}$} \\
|u_\delta(z) - y(p_\alpha^1) +  y(p_\alpha^1) - y(z')| & \\
\quad \leq 2\sqrt{2} L \big(|z-p_\alpha^1|+|p_\alpha^1-z'|\big) \leq 4L|z-z'|\big) & \text{if $z' \notin \mathcal{Q}^\mathrm{inner}$ and $p_\alpha^1 \notin \mathcal{Q}^\mathrm{inner}$} \\
|u_\delta(z) - y(p_\alpha^1)  + \tilde y(p_\alpha^1) - \tilde y(z') + y(p_\alpha^1) - \tilde y(p_\alpha^1) |  & \\
\quad \leq 2 \sqrt{2} L  \big(|z - p_\alpha^1| + |p_\alpha^1 - z'|  \big) + \frac{r}{6L^2}\, \frac{1}{2L} \leq 5L|z-z'| & \text{if $z' \notin \mathcal{Q}^\mathrm{inner}$ and $p_\alpha^1 \in \mathcal{Q}^\mathrm{inner}$} \\
\end{array} \right.
\end{align*} 
where we used that we already proved the bi-Lipschitz property inside the cross $Z_\alpha$ and in the second and fourth case we used that $\frac{r}{6L^2} \leq |p_\alpha^1-z'|$ since, in this cases, $p_\alpha^1$ and $z'$ have to lie on different edges. 

\begin{figure}[h!]
\centering
\includegraphics[width = 0.3 \paperwidth]{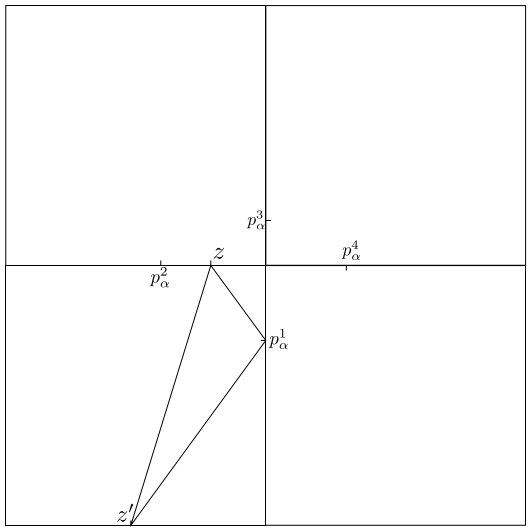}
\caption{The obtuse triangle formed by $z, z', p_\alpha^1$ as needed in Step 2}
\label{triangle}
\end{figure}

\vspace{2ex}

\noindent \emph{Step 4 of the proof of \eqref{biL}: Suppose that $z\in Z_\alpha$, $z' \in Z_\beta$ with $\alpha \neq \beta$}. \\
The last case we need to consider is when $z$, $z'$ lie in two crosses corresponding to two different vertices, respectively. In such a case $|w_\alpha - w_\beta|\geq r$ and also, from definition, $|u_\delta(z')-y(w_\beta)| \leq \frac{r}{4L}$ (as $z'$ belongs to the cross). Therefore,
$$
|y(w_\alpha) - u_\delta(z')| = |y(w_\beta) - y(w_\alpha) + u_\delta(z') - y(w_\beta)| \geq \frac{r}{L}-\frac{r}{4L} > \frac{r}{4L};
$$
i.e.\ $u_\delta(z') \notin \mathcal{B}(y(w_\alpha) ; \frac{r}{4L})$ and we may apply Remark \ref{BallArgument} to get
 (with $p_\alpha^1$ being the extremal of $Z_\alpha$ lying on the same edge as $z$)
$$
|u_\delta(z') - u_\delta(z)| \geq \frac{|u_\delta(p_\alpha^1)-u_\delta(z)| + |u_\delta(p_\alpha^1)-u_\delta(z')|}{3}.
$$
Similarly, also $u_\delta(p_\alpha^1) \notin \mathcal{B}(y(w_\beta) ; \frac{r}{4l})$ as
$$
|y(w_\beta) - y(w_\alpha) - u_\delta(p_\alpha^1) + y(w_\alpha)| \geq \frac{r}{L}-\frac{r}{4L} > \frac{r}{4L};
$$
and hence, again relying on Remark \ref{BallArgument} ($p_\beta^2$ denotes the extremal of $Z_\beta$ lying on the same edge as $z'$)
\begin{align*}
|u_\delta(z') - u_\delta(z)| &\geq \frac{|u_\delta(p_\alpha^1)-u_\delta(z)| + |u_\delta(p_\alpha^1)-u_\delta(p_\beta^2)| + |u_\delta(p_\beta^2)-u_\delta(z')|}{9} \\ &\geq \frac{1}{18L} \Big(|p_\alpha^1-z| + |p_\alpha^1-p_\beta^2| + |p_\beta^2-z'| \Big) \geq 
\frac{|z-z'|}{18L},
\end{align*}
by applying the triangle inequality. Moreover, we exploited that $|u_\delta(p_\alpha^1)-u_\delta(z)| \geq \frac{|p_\alpha^1-z|}{2L}$ as $p_\alpha^1$ and $z$ lie on the same edge within the same cross (cf.~Step 1); similarly also for $|u_\delta(p_\beta^2)-u_\delta(z')|$. Finally, we can see that $|u_\delta(p_\alpha^1)-u_\delta(p_\beta^2)| \geq \frac{|p_\alpha^1-p_\beta^2|}{2L}$ by the same procedure as employed in Step 3.

It, finally, remains to prove the upper bound in \eqref{biL}. But this follows from the fact that, since $z$,$z'$ belong to different crosses, there has to exist a point $p \in \mathcal{Q}$ that does not belong to any cross such that the triangle $z p z'$ is obtuse (or right) at $p$. Here, we admit also the extreme case in which $z p z'$ lie on a straight line; in this case, we understand the angle at $p$ to be $\pi$ and hence obtuse. Therefore, exploiting \eqref{ObtuseTriangle}, readily gives 
$$
|u_\delta (z) - u_\delta (p) + u_\delta (p)-u_\delta (z')| \leq 5L \big(|z-p|+|z'-p|\big) \leq \frac{10L}{\sqrt{2}}(|z-z'|) \leq 18L|z-z'|\ .
$$

\hfill $\Box$

\bigskip

\begin{remark}[Obtuse triangle inequality]
\label{ObtuseTriangle}
Let us consider a triangle formed by three points $z, p_1, z' \in \mathbb{R}^2$ such that the angle $\gamma$ at $p_1$ is obtuse or right (= larger or equal to $\pi/2$). Then it follows from the cosine law
\begin{align}
|z-z'| = \sqrt{|z-p_1|^2+|z'-p_1|^2 - 2|z-p_1|\, |z'-p_1|\cos(\gamma)} \geq \sqrt{|z-p_1|^2+|z'-p_1|^2}\nonumber\\ \geq 
\frac{\sqrt{2}}{2} \big( |z-p_1|+|z'-p_1|\big)\ .
\end{align}
\end{remark}

\begin{remark}[Ball separation inequality]
\label{BallArgument}
Let us consider a ball centered at $w$ with radius $\xi$ and a point $a$ lying inside this ball on the segment $w b$ with $|b-w|=\xi$. Moreover, let $c$ be a point lying outside this ball. Then, since $b$ is the nearest to $a$ lying on the boundary of the mentioned ball it has to hold that $|a-b| \leq |a-c|$ and so by the triangle inequality\footnote{Indeed $|b-c|\leq |a-b|+|a-c| \leq 2 |a-c|$ and so $|a-b|+|b-c| \leq 3|a-c|$ as desired.}
$$
|a-c| \geq \frac{|a-b|+|b-c|}{3}.
$$
\end{remark}

\bigskip

 \bigskip

{\bf Acknowledgment:}  We are indebted to two anonymous referees for many remarks, many useful suggestions, and for the extremely careful reading of the manuscript.  This work was supported by the GA\v{C}R grants P201/10/0357, P201/12/0671, P107/12/0121, 14-15264S, 14-00420S (GA\v{C}R), and the AV\v{C}R-DAAD project  CZ01-DE03/2013-2014.


\bigskip\bigskip
\vspace*{1cm}
\bigskip
\begin{thebibliography}{19}
\baselineskip=12pt
{\footnotesize
\bibitem{adams-fournier}
{\sc Adams, R.A., Fournier, J.J.F.}: {\it Sobolev spaces} 2nd ed., Elsevier, Amsterdam, 2003.


\bibitem{quasiregular}
{\sc Astala, K, Faraco, D}: Quasiregular mappings and Young measures. \emph{Proc. Royal Soc. Edinb. A} {\bf 132.05} (2002): 1045-1056.

\bibitem{ball77}
{\sc Ball, J.M.}: Convexity conditions and existence theorems in nonlinear elasticity. {\it Arch. Rat. Mech. Anal.} {\bf 63} (1977), 337--403.

\bibitem{ball81}
{\sc Ball, J.M.}: Global invertibility of Sobolev functions and the interpenetration of matter. {\it Proc.~Roy.~Soc.~Edinburgh } {\bf 88A} (1981), 315--328.



\bibitem{ball3}
{\sc Ball, J.M.}: A version of the fundamental theorem for Young measures. In:
{\it PDEs and Continuum Models of Phase Transition.} (Eds. M.Rascle, D.Serre,
M.Slemrod.) Lecture Notes in Physics {\bf 344}, Springer, Berlin, 1989,
pp.207--215.

\bibitem{ball-puzzles}
{\sc J.M. Ball}: {\it Some open problems in elasticity.} In Geometry, Mechanics, and Dynamics, pp.~3--59, Springer, New York, 2002.

\vspace{0mm}\bibitem{ball-james1} 
{\sc Ball, J.M., James, R.D.}: Fine phase mixtures as minimizers 
of energy. {\it Archive Rat. Mech. Anal.} {\bf 100} (1988), 13--52.

\bibitem{bbmkgpYm}
{\sc Bene\v{s}ov\'{a}, B., Kru\v{z}\'{\i}k, M., Path\'{o}, G.} Young measures supported on invertible matrices. {\it Appl. Anal.}  {\bf 93}  (2014), 105--123. 




\bibitem{ciarlet}
{\sc Ciarlet, P.G.}: {\it Mathematical Elasticity} Vol.~I: Three-dimensional Elasticity, North-Holland, Amsterdam, 1988.

\bibitem{ciarlet-necas}
{\sc Ciarlet P.G., Ne\v{c}as,  J.}: Injectivity and self-contact in nonlinear elasticity. {\it Arch. Rational Mech. Anal.} {\bf  97}
(1987),  171--188.

\bibitem{Conti}
{\sc Conti, S., Dolzmann, G.}: On the theory of relaxation in nonlinear elasticity with constraints on the determinant. \emph{arXiv preprint:1403.5779} (2014).

\bibitem{dacorogna}
{\sc Dacorogna, B.} {\it Direct Methods in the Calculus of Variations.} 2nd ed.~Springer,  2008.

\bibitem{daneri-pratelli-extension}
{\sc Daneri, S. Pratelli, A.}: A planar bi-Lipschitz extension theorem. Preprint, 2011.  (http://cvgmt.sns.it/paper/1675/).

\bibitem{daneri-pratelli}
{\sc Daneri, S. Pratelli, A.}: Smooth approximation of bi-Lipschitz orientation-preserving homeomorphisms. {\it Ann. Inst. H.~Poincar\'{e} An. Nonlin.} {\bf 31} (2014), 567--589.




\bibitem{fonseca-ganbo}
{\sc  Fonseca, I.,   Gangbo, W.}: \emph{ Degree Theory in Analysis and Applications},
Clarendon Press, Oxford, 1995.

\bibitem{evans-gariepy}
{\sc Evans, L.C., Gariepy, R.F.}: {\it Measure Theory and Fine Properties of Functions.} CRC Press, Inc. Boca Raton, 1992.
\bibitem{fonseca-leoni}
{\sc Fonseca, I., Leoni, G.}: {\it Modern Methods in the Calculus of Variations: $L^p$ Spaces}. Springer, 2007.


\bibitem{currents}
{\sc Giaquinta, M., Modica, G., Sou\v{c}ek, J.}: \emph{Cartesian currents in the calculus of variations.} (Vol. I, II). Springer, 1998.

\bibitem{Henao}
{\sc Henao, D., Mora-Corral, C.}: Invertibility and weak continuity of the determinant for the modelling of cavitation and fracture in nonlinear elasticity. \emph{Arch. Rat. Mech. Anal.} {\bf 197.2} (2010), 619--655.

\bibitem{iwaniec1}
{\sc T. Iwaniec, L.V. Kovalev, J. Onninen}: Diffeomorphic approximation of Sobolev homeomorphisms. \emph{Arch. Rat. Mech. Anal.} {\bf 201.3} (2011), 1047--1067.

\bibitem{k-p1}
{\sc Kinderlehrer, D., Pedregal, P.}: Characterization of Young measures generated by gradients. {\it Arch. Rat. Mech. Anal.} {\bf 115} (1991), 329--365.
\bibitem{k-p}
{\sc Kinderlehrer, D., Pedregal, P.}: Gradient Young measures
generated by sequences in Sobolev spaces. {\it J. Geom. Anal.} {\bf4}
(1994), 59--90.


\bibitem{krw}
{\sc Koumatos, K., Rindler, F., Wiedemann, E.}: Orientation-preserving Young measures. Preprint arXiv 1307.1007.v1, 2013.




\bibitem{kruzik-luskin}
{\sc Kru\v{z}\'{\i}k, M., Luskin, M.}: The computation of martensitic microstructure with piecewise laminates. {\it J. Sci. Comp.} {\bf 19} (2003), 293--308.

\bibitem{morrey}
{\sc Morrey, C.B.}: {\it Multiple Integrals in the Calculus of Variations.} Springer, Berlin, 1966.
\bibitem{mueller}
{\sc M\"{u}ller, S.}: {\it Variational models for microstructure and phase transisions.} Lecture Notes in Mathematics {\bf 1713}, Springer Berlin, 1999 pp. 85--210.
\bibitem{MullerQi}
{\sc M\"{u}ller, S., Tang, Q., Yan, B.S.}: On a new class of elastic deformations not allowing for cavitation. {\it Ann. Inst. H.~Poincar\'{e} An. Nonlin.} {\bf 11} (1994), 217--243.

\bibitem{MullerSpector}
{\sc M\"{u}ller, S, Spector, S.J.}: An existence theory for nonlinear elasticity that allows for cavitation. {\it Arch. Rat. Mech. Anal.} {\bf 131} (1995), 1--66.

\bibitem{pedregal}
{\sc Pedregal, P.}: {\it Parametrized Measures and Variational Principles.}
Birk\"auser, Basel, 1997.

\bibitem{r}
{\sc Roub\'{\i}\v{c}ek, T.}: {\it Relaxation in Optimization Theory and
Variational Calculus}. W. de Gruyter, Berlin, 1997.
\bibitem{schonbek}
{\sc Schonbek, M.E.}: Convergence of solutions to nonlinear dispersive
equations. {\it Comm. in Partial Diff. Equations} {\bf 7} (1982), 959--1000.
\bibitem{tang}
{\sc  Tang, Q.}: Almost-everywhere injectivity in nonlinear elasticity, \emph{Proc. Roy. Soc.
Edinburgh Sect.} {\bf A 109 } (1988), 79--95.

\bibitem{tartar}
{\sc Tartar, L.}: Beyond Young measures. {\it Meccanica} {\bf 30} (1995), 505-526.
\bibitem{tartar1}
{\sc Tartar, L.}: Mathematical tools for studying oscillations and concentrations: From Young measures to $H$-measures and their variants. In: {\it Multiscale problems in science and technology. Challenges to mathematical analysis and perspectives.} (N.~Antoni\v{c} et al. eds.) Proceedings of the conference on multiscale problems in science and technology, held in Dubrovnik, Croatia, September 3-9, 2000.  Springer, Berlin,  2002.

\bibitem{tukia-ext}
{\sc P. Tukia}: The planar Sch\"{o}nflies theorem for Lipschitz maps,\emph{ Ann. Acad. Sci. Fenn. Ser. A I Math} {\bf5} (1980), no. 1, 49-72.

\bibitem{y}
{\sc Young, L.C.}: Generalized curves and the existence of an attained
absolute minimum in the calculus of variations. {\it Comptes Rendus de la
Soci\'et\'e des Sciences et des Lettres de Varsovie}, Classe III {\bf 30}
(1937), 212--234.}
\end{thebibliography}
\end{document}